\documentclass{amsart}
\usepackage{amsmath}
\usepackage{amssymb}
\usepackage{amsthm}
\usepackage{graphicx}
\newcommand\sbullet[1][.6]{\mathbin{\vcenter{\hbox{\scalebox{#1}{$\bullet$}}}}}

\usepackage{biblatex}
\renewbibmacro{in:}{}

\theoremstyle{plain}
\newtheorem{theorem}{Theorem}[section]
\newtheorem{lemma}[theorem]{Lemma}
\newtheorem{proposition}[theorem]{Proposition}
\newtheorem{existence theorem}[theorem]{Existence Theorem}
\newtheorem{uniqueness theorem}[theorem]{Uniqueness Theorem}
\newtheorem{corollary}[theorem]{Corollary}
\newtheorem*{corollary3.10prime}{Corollary 3.10$^\prime$}

\theoremstyle{remark}
\newtheorem{remark}{Remark}

\title{A Pythagorean Theorem for volume}
\author{Fredric D. Ancel}
\address{Department of Mathematical Sciences, University of Wisconsin-Milwaukee}
\email{ancel@uwm.edu}
\date{May 13, 2023}
\bibliography{bibliography.bib}

\begin{document}
\maketitle

\begin{abstract}
    Lebesgue measurable subsets $A$ and $B$ of parallel or identical $k$-dimensional affine subspaces of Euclidean $n$-space $\mathbb{E}^n$ satisfy
    
    \medskip\noindent\textbf{\emph{The Product Formula for Volume:}}
    \begin{center}
    $Vol_{k}(A)Vol_{k}(B)=\sum_{J\in\mathcal{S}(n,k)}Vol_{k}(\pi_{J}(A))Vol_{k}(\pi_{J}(B))$.
    \end{center}
    
    \medskip\noindent Here $Vol_{k}$ denotes $k$-dimensional Lebesgue measure; $\mathcal{S}(n,k)$ denotes the set of all $k$-element subsets of $\{1,2,\cdots,n\}$; 
    and for $J\in\mathcal{S}(n,k)$, $\mathbb{E}^J=$\\
    $\{(x_1,x_2,\cdots,x_n)\in\mathbb{E}^n:x_i=0 \text{ for all } i\notin J\}$ and $\pi_{J}:\mathbb{E}^n\rightarrow\mathbb{E}^J$ is the projection that sends the $i^{th}$ coordinate of a point of $\mathbb{E}^n$ to $0$ whenever $i\notin J$.  Setting $B=A$, we obtain the corollary:

    \medskip\noindent\textbf{\emph{The Pythagorean Theorem for Volume:}}
    \begin{center}
    $Vol_{k}(A)^2=\sum_{J\in\mathcal{S}(n,k)}(Vol_{k}(\pi_{J}(A)))^2$.
    \end{center}
    
\end{abstract}

\section{Introduction}
If a vector space $V$ is equipped with an inner product $\langle\;,\;\rangle$, then this inner product satisfies the following equation.

\medskip\noindent\textbf{\emph{The Product Formula for Inner Products:}}
\begin{center}
$\langle x,y\rangle=\sum_{1\leq i\leq n}\langle x,u_i\rangle\langle y,u_i\rangle$
\end{center}

\medskip\noindent whenever $u_1,u_2,\cdots,u_n$ is an orthonormal basis for $V$. The fact that the dot product on Euclidean $n$-space $\mathbb{E}^n$ satisfies this equation with respect to the standard orthonormal basis for $\mathbb{E}^n$ allows us to deduce

\medskip\noindent\textbf{\emph{The Product Formula for Length:}}
\begin{center}
$Length(A)Length(B)=\sum_{1\leq j\leq n}Length(\pi_j(A))Length(\pi_j(B))$
\end{center}

\medskip\noindent provided that $A$ and $B$ are parallel or identical line segments in $\mathbb{E}^n$.  Here $\mathbb{E}^{\{j\}}=\{(x_1,x_2,\cdots,x_n)\in\mathbb{E}^n:x_i=0 \text{ for all }i\neq j\}$ and $\pi_j:\mathbb{E}^n\rightarrow\mathbb{E}^{\{j\}}$ is the projection that sends the $i^{th}$ coordinate of a point of $\mathbb{E}^n$ to $0$ whenever $i\neq j$. Setting $B=A$ then yields

\medskip\noindent\textbf{\emph{The Pythagorean Theorem for Length:}}
\begin{center}
$Length(A)^2=\sum_{1\leq j\leq n}(Length(\pi_j(A)))^2$
\end{center}

\medskip The object of this article is to exhibit an inner product on a vector space that allows the measurement of - not the \textit{length} of a line segment in $\mathbb{E}^n$ - but the volume of a subset of a k-dimensional affine subspace of $\mathbb{E}^n$.  Then, just as the dot product on $\mathbb{E}^n$ leads to the Product Formula and the Pythagorean Theorem for Length, this inner product yields the Product Formula and Pythagorean Theorem for $k$-dimensional Volume stated in the Abstract.  The emphasis in this article is conceptual rather than computational: the Product Formula and Pythagorean Theorem for Volume follow logically once Euclidean space with its dot product have been generalized to the appropriate inner product space.  The inner product space which is the appropriate generalization of $\mathbb{E}^n$ is a well-known algebraic object that is typically denoted $\Lambda_k(\mathbb{E}^n)$.  It is a vector subspace of the \textit{exterior algebra} of $\mathbb{E}^n$ which is usually denoted $\Lambda(\mathbb{E}^n)$.  In this article, we do not assume the readers are familiar with $\Lambda_k(\mathbb{E}^n)$ or its inner product.  We will provide a complete description of these objects.

The Product Formula and Pythagorean Theorem for Volume were new to the author when he first observed them.  Because of the remarkable way that the Pythagorean Theorem for Volume echoes the universally renowned formula for length, the author was surprised that the results for volume are not better known.  
However, the author’s subsequent search of the literature revealed that the Pythag-orean Theorem for Volume has been discovered, rediscovered, proved and reproved repeatedly.  
(The Product Formula for Volume, on the other hand, is not explicitly stated in any source found by the author; this article may be its first formulation.)
Apparently, the earliest statement of the Pythagorean Theorem for the area of a triangle in $\mathbb{E}^3$ is found in the 1220 book \emph{Practica geometriae} by L. Fibonacci (1172-1250) (\cite{boyer} page 232).  
A typical application of this result would tell us that for a triangle $T$ in $\mathbb{E}^3$ with vertices $P=(p,0,0)$, $Q=(0,q,0)$ and $R=(0,0,r)$ where $p,\;q\text{ and }r > 0$, $(Area(T))^2=(\frac{1}{2}qr)^2+(\frac{1}{2}pr)^2+(\frac{1}{2}pq)^2$, because the projections of $T$ into the $YZ$-plane, the $XZ$-plane and the $XY$-plane are triangles that have areas $\frac{1}{2}qr$, $\frac{1}{2}pr$ and $\frac{1}{2}pq$, respectively.\footnote{To corroborate this formula for $Area(T)$, note that $Area(T)=\frac{1}{2}Area(Z)$ where $Z$ is the parallelogram with vertices $P$, $Q$, $R$ and $Q+R-P$ and $Area(Z)$ equals the norm of the cross product $(Q-P)\times(R-P)$.}  
According to page 37 of \cite{eves} and various Wikipedia articles, following its appearance in Fibonacci’s book, this result (for a triangle in $\mathbb{E}^3$) was familiar to R. Descartes (1596-1650) and J. Faulhaber (1580-1635).  
Two independent proofs of this result were published in the late $18^{th}$ century, one in 1772 by C. de Tinseau (1748-1822) and the other in 1783 by J. P. de Gua de Malves (1712-1785).  
Ironically, the Pythagorean Theorem for the area of a triangle in $\mathbb{E}^3$ then acquired the name \emph{de Gua’s Theorem}.  
A statement and careful proof of the Pythagorean Theorem for the area of a planar polygonal figure in $\mathbb{E}^3$ appeared in the influential 1813 memoir ~\cite{hachette} by J. Hachette (1769-1834) and G. Monge (1746-1818) (See \emph{Theoreme 66} on page 100 of \cite{hachette}, also cited on page 436 of \cite{boyer}).
According to \cite{knill}, in about 1812, J. Binet (1786-1856) and A-L. Cauchy (1789-1857) independently discovered a formula (named both \emph{Binet-Cauchy} and \emph{Cauchy-Binet}) for the determinant of the product of two non-square matrices whose product is a square matrix. 
(Binet was the inventor of matrix multiplication.)
This formula, if interpreted correctly, is equivalent to the Product Formula for the inner product on $\Lambda_k(\mathbb{E}^n)$.  
Of the nineteen articles cited in the references, eleven of them - \cite{alvarez}, \cite{atzema}, \cite{cho}, \cite{cook}, \cite{czyzewska}, \cite{donchian}, \cite{drucker}, \cite{fitting}, \cite{jacoby}, \cite{lin} and \cite{porter} - are recapitulations of the Pythagorean Theorem for Volume, in some instances specialized to parallelograms in $\mathbb{E}^3$, and in other instances covering the more general case of $k$-dimensional parallelopipeds in $\mathbb{E}^n$. 
Of these, references \cite{alvarez} and \cite{jacoby} contain informative historical remarks.  
Paper \cite{czyzewska} is the only one found by the author that makes the observation that the Pythagorean Theorem for Volume applies to all Lebesgue measurable subsets of an affine subspace of $\mathbb{E}^n$. 
The author was unaware of the existence of \cite{czyzewska} when he discovered the results stated in the Abstract; and the methods of this article are more linear algebraic and less analytic in character than the methods of \cite{czyzewska}.  
Since the author’s search for papers on the Pythagorean Theorem for Volume was more haphazard than exhaustive, there are no doubt many such articles that are unintentionally omitted from the references.

\section{Linear algebraic preliminaries}
In this article, we restrict ourselves to vector spaces and matrices over the field $\mathbb{R}$ of real numbers.

Let $V$ be a vector space.
For $p\in V$, define the \emph{translation} $T_p:V\rightarrow V$ by $T_p(x)=x+p$.
Note that for $p$, $q\in V$, $T_q\circ T_p=T_{p+q}$ and $T_p^{-1}=T_{-p}$.
A subset $W$ of $V$ is a \emph{$k$-dimensional affine subspace} if there is a $k$-dimensional vector subspace $W_0$ of $V$ and a $p\in V$ such that $W=T_p(W_0)$.
Two $k$-dimensional affine subspaces $W$ and $W'$ of $V$ are \emph{identical or parallel} if there is a $q\in V$ such that $W'=T_q(W)$.  
Observe that if $W_0$ is a vector subspace of $V$ and $p$, $q\in V$, then $T_p(W_0)=T_q(W_0)$ if and only if $p-q\in W_0$.

\begin{proposition}
If $L:V\rightarrow V'$ is a linear function between vector space, then for every $p\in V$, $L\circ T_p=T'_{L(p)}\circ L$.  (Here, for $p\in V$ and $p'\in V'$, $T_p:V\rightarrow V$ and $T'_{p'}:V'\rightarrow V'$ denote translations.)
\end{proposition}

\begin{proof}
For every $x\in V$, $L\circ T_p(x)=L(x+p)=L(x)+L(p)=T'_{L(p)}\circ L(x).$
\end{proof}

We call a vector space an \emph{inner product space} if it is equipped with an inner product. 
If $V$ is an inner product space equipped with the inner product $\langle$ , $\rangle_V$, then define the associated norm $||$ $||_V:V\rightarrow [0,\infty)$ by $||x||_V=(\langle x,x\rangle_V)^{\frac{1}{2}}$ and define the associated metric $\rho_V:V\times V\rightarrow [0,\infty)$ by $\rho_V(x,y)=||x-y||_V$.
The metric $\rho_V$ determines a topology on $V$.  
If $V$ and $W$ are inner product spaces, then a function $f:V\rightarrow W$ is an \emph{isometry} if $f(V)=W$ and $\rho_W(f(x),f(y))=\rho_V(x,y)$ for all $x,y\in V$.
Note that the latter condition is equivalent to the statement $||f(x)-f(y)||_W=||x-y||_V$ for all $x, y \in V$. 
Observe that if $V$ is an inner product space and $p\in V$, then the translation $T_p:V\rightarrow V$ is an isometry.

For a vector space $V$ and an integer $k\geq 1$, let $V^k=\{(v_1,v_2,\cdots,v_k):v_i\in V$ for $1\leq i\leq k\}$.  Let $V$ and $W$ be vector spaces, let $k\geq 1$ and let $f:V^k\rightarrow W$ be a function.
f is \emph{multilinear} if
\[
f(v_1,\cdots,v_{i-1},av_i+bw_i,v_{i+1},\cdots,v_k)=
\]
\[
af(v_1,\cdots,v_{i-1},v_i,v_{i+1},\cdots,v_k)+bf(v_1,\cdots,v_{i-1},w_i,v_{i+1},\cdots,v_k)
\]
for all $1\leq i\leq k$, all $v_1,\cdots,v_{i-1},v_i,w_i,v_{i+1},\cdots,v_k\in V$ and all $a,b\in \mathbb{R}$.  When $k=2$ and $f$ is multilinear, then we say that $f$ is \emph{bilinear}.  f is \emph{alternating} if
\[
f(v_1,\cdots,v_{i-1},v_j,v_{i+1},\cdots,v_{j-1},v_i,v_{j+1},\cdots,v_k)=
\]
\[
-f(v_1,\cdots,v_{i-1},v_i,v_{i+1},\cdots,v_{j-1},v_j,v_{j+1},\cdots,v_k)
\]
for all $1\leq i<j\leq k$ and all $v_1,\cdots,v_{i-1},v_i,v_{i+1},\cdots,v_{j-1},v_j,v_{j+1},\cdots,v_k\in V$.  (Equivalently, f is alternating if 
\[
f(v_1,\cdots,v_{i-1},v_i,v_{i+1},\cdots,v_{j-1},v_j,v_{j+1},\cdots,v_k)=0
\]
whenever $v_i=v_j$.)

Let $V$ be a vector space and let $k\geq 1$ be an integer.
Then there is a vector space $\Lambda_k(V)$ and a multilinear map $\textit{i}:V^k\rightarrow\Lambda_k(V)$ with the following \emph{universality property}.
If $f:V^k\rightarrow W$ is a multilinear alternating map to a vector space $W$, then there is a unique linear map $\Phi:\Lambda_k(V)\rightarrow W$ such that $\Phi\circ\textit{i}= f$.
It is common practice to denote the image $\textit{i}(v_1,v_2,\cdots,v_k)$ by $v_1\wedge v_2\wedge\cdots\wedge v_k$ and to call this point a \emph{$k$-vector} or a \emph{wedge product} of $k$ elements of $V$.

For a textbook treatment of $\Lambda_k(V)$, we refer the reader to pages 446-449 of \cite{dummit} and to \cite{wiki}.  
In those pages, it is revealed that $\Lambda_k(V)$ is a vector subspace of a larger vector space $\Lambda(V)$ called the \emph{exterior algebra} of $V$.  
$\Lambda(V)$ is the direct sum of its subspaces $\{\Lambda_k(V):k\geq 0\}$, and $\Lambda(V)$ is called an \emph{algebra} because there is a product $(x,y)\mapsto x\wedge y:\Lambda(V)\times\Lambda(V)\rightarrow\Lambda(V)$ which maps $\Lambda_k(V)\times\Lambda_m(V)$ into $\Lambda_{k+m}(V)$.  
We will now list the facts about $\Lambda_k(V)$ that we will need in the sequel, and we will justify those facts when the justifications aren’t obvious.  
The reader can consult \cite{dummit} and \cite{wiki} for a broader and more detailed overview of this topic.

The existence of $\Lambda_k(V)$ can be proved by identifying it with the quotient space $\Phi(V^k)/K$ where $\Phi(V^k)$ is a vector space with basis $V^k$ and $K$ is the vector subspace of $\Phi(V^k)$ generated by all elements of $\Phi(V^k)$ of the form
\[
(v_1,\cdots,v_{i-1},av_i+bw_i,v_{i+1},\cdots,v_k)
\]
\[
-a(v_1,\cdots,v_{i-1},v_i,v_{i+1},\cdots,v_k)-b(v_1,\cdots,v_{i-1},w_i,v_{i+1},\cdots,v_k)
\]
where $1\leq i\leq k$, $v_1,\cdots,v_{i-1},v_i,w_i,v_{i+1},\cdots,v_k\in V$ and $a,b\in \mathbb{R}$ together with all elements of $\Phi(V^k)$ of the form
\[
(v_1,\cdots,v_i,\cdots,v_j,\cdots,v_k)
\]
where $1\leq i<j\leq k$, $v_1,\cdots,v_i,\cdots,v_j,\cdots,v_k\in V$ and $v_i=v_j$.\footnote{For a set $X$, a vector space $\Phi(X)$ with basis $X$ can be obtained by setting $\Phi(X)$ equal to the set of all functions $\phi:X\rightarrow\mathbb{R}$ such that $\{x\in X:\phi(x)\neq0\}$ is finite, and endowing $\Phi(X)$ with the operations of pointwise addition and pointwise scalar multiplication.  
For each $x\in X$, define $\phi_x\in\Phi(X)$ by $\phi_x(x)=1$ and 
$\phi_x(y)=0$ for $y\in X-\{x\}$.  
Then $\{\phi_x:x\in X\}$ is a basis for $\Phi(X)$.  
By identifying each $x\in X$ with $\phi_x\in\Phi(X)$, we make $X$ a basis for $\Phi(X)$.}

We make two additional remarks about $\Lambda_k(V)$.
\begin{remark}
The uniqueness assertion in the universality property stated in the definition of $\Lambda_k(V)$ implies that the set $\textit{i}(V^k)=\{v_1\wedge v_2\wedge\cdots\wedge v_k:(v_1,v_2,\cdots,v_k)\in V^k\}$ generates $\Lambda_k(V)$.
\end{remark}
\begin{remark}
Suppose $(u_1,u_2,\cdots,u_n)$ is an ordered basis for $V$ and $1\leq k\leq n$.
For $J\in\mathcal{S}(n,k)$, let $1\leq j_1<j_2<\cdots<j_k\leq n$ so
that $J=\{j_1,j_2,\cdots,j_k\}$, let $\textbf{u}_J=(u_{j_1},u_{j_2},\cdots,u_{j_k})$ and let $\wedge_J\textbf{u}=\textit{i}(\textbf{u}_J)=u_{j_1}\wedge u_{j_2}\wedge\cdots\wedge u_{j_k}$.
Then $\{\wedge_J\textbf{u}:J\in\mathcal{S}(n,k)\}$ is a basis for $\Lambda_k(V)$.
Hence, if $1\leq k\leq n$ and $dim(V)=n$, then $dim(\Lambda_k(V))=\binom{n}{k}$.
The universality property is used to prove that $\{\wedge_J\textbf{u}:J\in\mathcal{S}(n,k)\}$ is linearly independent.
\end{remark}

Now suppose that $V$ is an $n$-dimensional inner product space with inner product $\langle\;,\;\rangle$ and $1\leq k\leq n$.
We will define an inner product $\langle\langle\;,\;\rangle\rangle$ on $\Lambda_k(V)$ with the property that if $(u_1,u_2,\cdots,u_n)$ is an ordered orthonormal basis for $V$, then $\{\wedge_J\textbf{u}:J\in\mathcal{S}(n,k)\}$ is an orthonormal basis for $\Lambda_k(V)$.
Define $\prec\;,\;\succ:V^k\times V^k\rightarrow\mathbb{R}$ by $\prec(v_1,v_2,\cdots,v_k),(w_1,w_2,\cdots,w_k)\succ\;=Det(\langle v_i,w_j\rangle)$.
Observe that $\prec\;,\;\succ$ is multilinear in the sense that it is linear in each of the variables $v_1,v_2,\cdots,v_k,w_1,w_2,\cdots,w_k$.
Also observe that $\prec\;,\;\succ$ is \emph{bi-alternating}; 
in other words, $\prec\;,\;\succ$ is alternating separately in the variables $v_1,v_2,\cdots,v_k$ and in the variables $w_1,w_2,\cdots,w_k$; that is, $\prec(v_1,v_2,\cdots,v_k),(w_1,w_2,\cdots,w_k)\succ\;=0$ whenever either $v_i=v_j$ or $w_i=w_j$ for $1\leq i < j\leq k$.
Further observe that if $(u_1,u_2,\cdots,u_n)$ is an ordered basis for $V$, then for $J,K\in\mathcal{S}(n,k)$: $\prec\textbf{u}_J,\textbf{u}_K\succ\;=1$ if $J=K$ and $\prec\textbf{u}_J,\textbf{u}_K\succ\;=0$ if $J\neq K$.
To ``lift" $\prec\;,\;\succ$ to an inner product on $\Lambda_k(V)$, we require the following generalization of the universality property.

\begin{lemma}
If $\alpha:V^k\times V^k\rightarrow W$ is a multilinear bi-alternating map to a vector space W, then there is a unique bilinear map $\beta:\Lambda_k(V)\times\Lambda_k(V)\rightarrow W$ such that $\beta\circ(\textit{i}\times \textit{i})=\alpha$.
\end{lemma}

\begin{proof}
For each $\textbf{w}\in V^k$, define $\alpha_{\textbf{w}}:V_k\rightarrow W$ by $\alpha_{\textbf{w}}(\textbf{v})=\alpha(\textbf{v},\textbf{w})$.  
Then each $\alpha_{\textbf{w}}$ is a multilinear alternating map.  
Hence, for each $\textbf{w}\in V^k$, the universality property provides a unique linear map $\phi_{\textbf{w}}:\Lambda_k(V)\rightarrow W$ such that 
$\phi_{\textbf{w}}\circ\textit{i}=\alpha_{\textbf{w}}$.  
Let $\mathcal{L}(\Lambda_k(V),W)$ denote the set of all linear maps from $\Lambda_k(V)$ to $W$.  
Regard $\mathcal{L}(\Lambda_k(V),W)$ as a vector space with respect to the operations of pointwise addition and pointwise scalar multiplication.  
Define $\Phi:V^k\rightarrow\mathcal{L}(\Lambda_k(V),W)$ by $\Phi(\textbf{w})=\phi_{\textbf{w}}$.  
It is easy to verify that $\Phi$ is multilinear and alternating.  
Therefore the universality property provides a unique linear map $\Psi:\Lambda_k(V)\rightarrow\mathcal{L}(\Lambda_k(V),W)$ such that $\Psi\circ\textit{i}=\Phi$. 
Define $\beta:\Lambda_k(V)\times\Lambda_k(V)\rightarrow W$ by $\beta(x,y)=\Psi(y)(x)$.  
$\beta$ is clearly bilinear and clearly satisfies $\beta\circ(\textit{i}\times\textit{i})=\alpha$.  
The uniqueness of $\beta$ follows from the fact that $\textit{i}(V^k)$ generates $\Lambda_k(V).$ 
\end{proof}

This lemma implies that there is a bilinear map $\langle\langle\;,\;\rangle\rangle:\Lambda_k(V)\times\Lambda_k(V)\rightarrow \mathbb{R}$ such that $\langle\langle\;,\;\rangle\rangle\circ(\textit{i}\times\textit{i})=\prec\;,\;\succ$.  
Hence, if $(u_1,u_2,\cdots,u_n)$ is an ordered orthonormal basis for $V$, then for $J,K\in\mathcal{S}(n,k)$: $\langle\langle\wedge_J\textbf{u},\wedge_K\textbf{u}\rangle\rangle=\;\prec\textbf{u}_J,\textbf{u}_K\succ\;=1$ if $J=K$ and $\langle\langle\wedge_J\textbf{u},\wedge_K\textbf{u}\rangle\rangle=\;\prec\textbf{u}_J,\textbf{u}_K\succ\;=0$ if $J\neq K$.
Thus, $\{\wedge_J\textbf{u}:J\in\mathcal{S}(n,k)\}$ is orthonormal with respect to $\langle\langle\;,\;\rangle\rangle$.  
It remains to prove that $\langle\langle\;,\;\rangle\rangle$ is positive definite.  
This follows easily because $\{\wedge_J\textbf{u}:J\in\mathcal{S}(n,k)\}$ is a basis for $\Lambda_k(V)$ that is orthonormal with respect to $\langle\langle\;,\;\rangle\rangle$ and because $\langle\langle\;,\;\rangle\rangle$ is bilinear.  
These observations combined with the Remark 2 above  tell us that $\{\wedge_J\textbf{u}:J\in\mathcal{S}(n,k)\}$ is an orthonormal basis for $\Lambda_k(V)$ with respect to the inner product $\langle\langle\;,\;\rangle\rangle$.

We finish this section by establishing some matrix notation that will be used later.  
Suppose $\textbf{u}=(u_1,u_2,\cdots,u_k)$ is an ordered basis for a vector space $V$ and $\textbf{v}=(v_1,v_2,\cdots,v_m)$ is an m-tuple of elements of $V$.  
Let $C_{\textbf{u}}(\textbf{v})$ be the $m\times k$ matrix $(c_{i,j})$ determined by the equations 
\[
v_i=\sum_{1\leq j\leq k}c_{i,j}u_j
\]
for $1\leq i\leq m$, and call $C_{\textbf{u}}(\textbf{v})$ the \emph{coordinate matrix of $\textbf{v}$ with respect to $\textbf{u}$}. 
(The $i^{th}$ row of $C_{\textbf{u}}(\textbf{v})$ lists the coordinates of $v_i$ with respect to \textbf{u}.)   
Note that since $u_1,u_2,\cdots,u_k$ are linearly independent, then $C_{\textbf{u}}(\textbf{v})$ is uniquely determined by \textbf{u} and \textbf{v}.  
Observe that $C_{\textbf{u}}(\textbf{u})$ is the $k\times k$ identity matrix. 

Suppose $\textbf{u}=(u_1,u_2,\cdots,u_k)$ and $\textbf{v}=(v_1,v_2,\cdots,v_m)$ are ordered bases for vector spaces $U$ and $V$, respectively, and $L:U\rightarrow V$ is a linear map.  
Let $C_{\textbf{u},\textbf{v}}(L)$ be the $k\times m$ matrix
$(a_{i,j})$ determined by the equations
\[
L(u_i)=\sum_{1\leq j\leq m}a_{i,j}v_j
\]
for $1\leq i\leq k$, and call $C_{\textbf{u},\textbf{v}}(L)$ the \emph{coordinate matrix of $L$ with respect to \textbf{u} and \textbf{v}}.  
(The $i^{th}$ row of $C_{\textbf{u},\textbf{v}}(L)$ lists the coordinates of $L(\textbf{u}_i)$ with respect to \textbf{v}.)   
As above, $C_{\textbf{u},\textbf{v}}(L)$ is uniquely determined by $\textbf{u}$, $\textbf{v}$ and $L$ because $v_1,v_2,\cdots,v_m$ are linearly independent.  
We state without proof:
\begin{proposition}
    If $\textbf{u}=(u_1,u_2,\cdots,u_k)$ and $\textbf{v}=(v_1,v_2,\cdots,v_m)$ are ordered bases for vector spaces $U$ and $V$, respectively, $L:U\rightarrow V$ is a linear map, $\textbf{w}=(w_1,w_2,\cdots,w_r)$ is an $r$-tuple of elements of $U$, and 
\[
    L(\textbf{w})=(L(w_1),L(w_2),\cdots,L(w_r)),
\] 
    then $C_{\textbf{v}}(L(\textbf{w}))=C_{\textbf{u}}(\textbf{w})C_{\textbf{u},\textbf{v}}(L)$.
    In particular, if $V=U$ and $\textbf{w}=\textbf{v}=\textbf{u}$, then $C_{\textbf{u}}(L(\textbf{u}))=C_{\textbf{u}}(\textbf{u})C_{\textbf{u},\textbf{u}}(L)=C_{\textbf{u},\textbf{u}}(L)$.  \qed
\end{proposition}

If $A=(a_{i,j})$ is a $k\times m$ matrix, recall that the \emph{transpose} of $A$ is the $m\times k$ matrix $A^T=(a_{i,j}^T)$ defined by $a_{i,j}^T=a_{j,i}$ for $1\leq i\leq m, 1\leq j\leq k$.
Also recall that if $A$ is a $k\times k$ matrix, then $Det(A^T)=Det(A)$.

\section{Measure theoretic preliminaries}
$k$-dimensional Lebesgue measure is typically developed for subsets of $\mathbb{E}^k$.
For the purposes of this article, it is more useful to formulate $k$-dimensional Lebesgue measure in a slightly more general context - for subsets of $k$-dimensional inner product spaces.
Therefore, we will state the basic definitions and relevant fundamental results of measure theory in this more general setting.
The proofs of the theorems about Lebesgue measure stated in this section, at least in the setting of $\mathbb{E}^k$, can be found in the standard texts on measure theory, and references to these proofs are provided.  The process of generalizing these proofs to the setting of inner product spaces is, for the most part, routine and is left to the reader.

A collection $\mathcal{M}$ of subsets of a set $X$ is a \emph{$\sigma$-algebra} on $X$ if it satisfies the following three properties:  
\textit{i})\;$\emptyset\in\mathcal{M}$.  
\textit{ii})\;$\mathcal{M}$ is \emph{closed under the formation of countable unions;} 
i.e., if $\mathcal{A}$ is a countable subcollection of $\mathcal{M}$, then $\bigcup\mathcal{A}\in\mathcal{M}$.  
\textit{iii})\;$\mathcal{M}$ is \emph{closed under the formation of complements;} 
i.e., if $A\in\mathcal{M}$, then $X-A\in\mathcal{M}$.
It follows easily that every $\sigma$-algebra $\mathcal{M}$ on X also satisfies the following properties:
\textit{iv})\;$\mathcal{M}$ is \emph{closed under the formation of differences;}
i.e., if $A,B\in\mathcal{M}$, then $A-B\in\mathcal{M}$.
\textit{v})\;$\mathcal{M}$ is \emph{closed under the formation of countable intersections;}
i.e., if $\mathcal{A}$ is a countable subcollection of $\mathcal{M}$, then $\bigcap\mathcal{A}\in\mathcal{M}$.

A \emph{measure} on a set $X$ is a function $\mu:\mathcal{M}\rightarrow[0,\infty]$ where $\mathcal{M}$ is a $\sigma$-algebra on $X$ and $\mu$  satisfies the following two properties: 
\textit{i})\;$\mu(\emptyset)=0$.
\textit{ii})\;$\mu$ is \emph{countably additive};
i.e., if $\mathcal{A}$ is a countable pairwise disjoint subcollection of $\mathcal{M}$, then $\mu(\bigcup\mathcal{A})=\sum_{A\in\mathcal{A}}\mu(A)$.
It is easily proved that every measure $\mu$ satisfies the following additional properties:
\textit{iii})\;$\mu$ is \emph{monotone};
i.e., if $A,B\in\mathcal{M}$ and $A\subset B$, then $\mu(A)\leq\mu(B)$.
\textit{iv})\;$\mu$ is \emph{countably subadditive};
i.e., if $\mathcal{A}$ is a countable (not necessarily pairwise disjoint) subcollection of $\mathcal{M}$, then $\mu(\bigcup\mathcal{A})\leq\sum_{A\in\mathcal{A}}\mu(A)$.
\textit{v})\;If $A_1\subset A_2\subset A_3\subset\cdots$ is an increasing sequence of elements of $\mathcal{M}$, then $\mu(\bigcup_{i\geq 1}A_i)=sup\{\mu(A_i):i\geq 1\}$.
\textit{vi})\;If $A_1\supset A_2\supset A_3\supset\cdots$ is a decreasing sequence of elements of $\mathcal{M}$ and $\mu(A_1)<\infty$, then $\mu(\bigcap_{i\geq 1}A_i)=inf\{\mu(A_i):i\geq 1\}$.

A measure $\mu:\mathcal{M}\rightarrow[0,\infty]$ on a set $X$ is \emph{complete} if every subset of a measure 0 element of $\mathcal{M}$ is an element of $\mathcal{M}$. 
A simple argument shows that if $\mu:\mathcal{M}\rightarrow[0,\infty]$ is a complete measure on $X$, then for $A\subset B\subset C\subset X$, if $A,C\in\mathcal{M}$ and $\mu(A)=\mu(C)$, then $B\in\mathcal{M}$.

A $\sigma$-algebra $\mathcal{M}$ on a topological space $X$ is called a \emph{Borel $\sigma$-algebra} if $\mathcal{M}$ contains every open subset of X.  A measure $\mu:\mathcal{M}\rightarrow[0,\infty]$ on a topological space is called a \emph{Borel measure} if $\mathcal{M}$ us a Borel $\sigma$-algebra.

A measure $\mu:\mathcal{M}\rightarrow[0,\infty]$ on a topological space $X$ is \emph{regular} if $\mathcal{M}$ is a Borel $\sigma$-algebra and for every $A\in\mathcal{M}$,
\[
\mu(A)=inf\{\mu(U):U\text{ is an open subset of }X\text{ and }A\subset U\}
\]
\[
=sup\{\mu(C):C\text{ is a closed subset of }X\text{ and }C\subset A\}. 
\]

A measure $\mu:\mathcal{M}\rightarrow[0,\infty]$ on a vector space $V$ is \emph{translation invariant} if for every $A\in\mathcal{M}$ and every $p\in V$, $\mu(T_p(A))=\mu(A)$. 
(Recall that $T_p:V\rightarrow V$ is the translation specified by $T_p(x)=x+p$.)

A \emph{pseudometric} on a set $X$ is a function $\rho:X\times X\rightarrow [0,\infty]$ which for all $x,y\text{ and }z\in X$ has the following three properties: 
\textit{i})\;$\rho(x,x)=0$.
\textit{ii})\;$\rho(x,y)=\rho(y,x)$.
\textit{iii})\;$\rho$ satisfies the \emph{triangle inequality}; i.e., $\rho(x,z)\leq\rho(x,y)+\rho(y,z)$.
If $\rho$ is a pseudometric on $X$, then for every $x\in X$ and every $\epsilon>0$, let $\mathcal{N}_{\rho}(x,\epsilon)=\{y\in X:\rho(x,y)<\epsilon\}$.
Then $\{\mathcal{N}_{\rho}(x,\epsilon):x\in X\text{ and }\epsilon>0\}$ is a basis for a topology on $X$ called the \emph{topology on $X$ determined by $\rho$}.  
If $\rho$ is a pseudometric on a set $X$, then the pair 
$(X,\rho)$ is called a \emph{pseudometric space}.

For two sets $A$ and $B$, the \emph{symmetric difference of $A$ and $B$} is the set $A\Delta B=(A-B)\cup(B-A)$.

Let $\mu:\mathcal{M}\rightarrow[0,\infty]$ be a measure on a set $X$.  
A pseudometric $\rho$ on $\mathcal{M}$, called the \emph{symmetric difference pseudometric associated with $\mu$}, is defined by $\rho(A,B)=\mu(A\Delta B)$.
(The verification that $\rho$ is a pseudometric on $\mathcal{M}$ is straightforward. 
The verification of the triangle inequality depends on the fact that for $A,B,C\in\mathcal{M},A\Delta C\subset(A\Delta B)\cup(B\cup C)$.)
We note that $\rho$ satifies the inequality $|\mu(A)-\mu(B)|\leq\rho(A\Delta B)$.
(\emph{Proof.} $|\mu(A)-\mu(B)|=|(\mu(A-B)+\mu(A\cap B))-(\mu(B-A)+\mu(A\cap B))|=|\mu(A-B)-\mu(B-A)|\leq\mu(A-B)+\mu(B-A)=\rho(A,B)$. \qedsymbol)
It follows that if $\mathcal{M}$ is equipped with the topology determined by $\rho$, then $\mu:\mathcal{M}\rightarrow[0,\infty]$ is continuous.

Let $V$ be a $k$-dimensional inner product space, let $\textbf{u}=(u_1,u_2,\cdots,u_k)$ be an ordered orthonormal basis for $V$, and let $[\textbf{u}]=[u_1,u_2,\cdots,u_k]=\{\sum_{1\leq i\leq k}t_iu_i:0\leq t_i\leq1\text{ for }1\leq i\leq k\}$.  
A measure $\mu:\mathcal{M}(V,\textbf{u})\rightarrow [0,\infty]$ is called a \emph{$k$-dimensional Lebesgue \textbf{u}-measure on V} if $\mu$ is a complete, Borel, regular, translation invariant measure such that $\mu([\textbf{u}])=1$.

A measure $\mu:\mathcal{M}(V,\textbf{u})\rightarrow [0,\infty]$ on a set X is \emph{$\sigma$-finite} if there is a countable subset $\mathcal{A}$ of $\mathcal{M}$ such that $\bigcup\mathcal{A}=X$ and $\mu(A)<\infty$ for every $A\in\mathcal{A}$.  
If $V$ is a $k$-dimensional inner product space and \textbf{u} is an ordered orthonormal basis for $V$, then every $k$-dimensional Lebesgue \textbf{u}-measure $\mu$ on $V$ is $\sigma$-finite.
Indeed, if $P=\{\sum_{1\leq i\leq k}n_iu_i:n_i\text{ is an integer for }1\leq i\leq k\}$, then $V$ is the union of the countable collection $\{T_p([\textbf{u}]):p\in P\}$ and $\mu(T_p([\textbf{u}]))=\mu([\textbf{u}])=1$ for each $p\in P$.

\begin{existence theorem}
If $V$ is a $k$-dimensional inner product space and \textbf{u} is an ordered orthonormal basis for $V$, then a $k$-dimensional Lebesgue \textbf{u}-measure on $V$ exists.
\end{existence theorem}

The statement and a proof of the Existence Theorem can be cobbled together from Theorems 20.1.11, 20.1.13 and Problem 20.2.20 on pages 426-427, 427-428 and 435, respectively, of \cite{royden}.

\begin{uniqueness theorem}
Suppose $V$ is a $k$-dimensional inner product space, \textbf{u} is an ordered orthonormal bases for $V$ and $\mu:\mathcal{M}\rightarrow[0,\infty]$ is a $k$-dimensional Lebesgue \textbf{u}-measure on $V$.  
If $\mu^\prime:\mathcal{M}^\prime\rightarrow[0,\infty]$ is a Borel, translation invariant measure on V such that $\mu^\prime(U_0)<\infty$ for some non-empty open subset $U_0$ of $V$, then either $\mu^\prime=0$ or there is a $\kappa\in(0,\infty)$ such that $\mu^\prime\vert\mathcal{M}\cap\mathcal{M}^\prime=\kappa\mu\vert\mathcal{M}\cap\mathcal{M}^\prime$.
If, in addition, $\mu^\prime$ is non-zero, complete and regular, then $\mathcal{M}^\prime=\mathcal{M}$.
\end{uniqueness theorem}

A proof of the Uniqueness Theorem 3.2 can be derived, with some effort, from a result known as the Carath\'{e}odory-Hahn Theorem (pages 356-357 of \cite{royden}).

We note two important consequences of the Uniqueness Theorem.  First, if \textbf{u} is an ordered orthonormal basis for a $k$-dimensional inner product space $V$ and $\mu:\mathcal{M}\rightarrow[0,\infty]$ and $\mu^\prime:\mathcal{M}^\prime\rightarrow[0,\infty]$ are both $k$-dimensional Lebesgue \textbf{u}-measures on $V$, then $\mathcal{M}=\mathcal{M}^\prime$ and $\mu=\mu^\prime$.
Indeed, the Uniqueness Theorem implies $\mathcal{M}^\prime=\mathcal{M}$ and $\mu^\prime=\kappa\mu$ for some $\kappa\in(0,\infty)$.
However, $\kappa=\kappa\mu([\textbf{u}])=\mu^\prime([\textbf{u}])=1$.
In other words, for every ordered orthonormal basis \textbf{u} for V, there is \emph{one and only one} $k$-dimensional Lebesgue \textbf{u}-measure on $V$.
Second, if \textbf{u} and $\textbf{u}^\prime$ are ordered orthonormal bases for the $k$-dimensional inner product space $V$, $\mu:\mathcal{M}\rightarrow[0,\infty]$ is a $k$-dimensional Lebesgue \textbf{u}-measure on $V$ and $\mu^\prime:\mathcal{M}^\prime\rightarrow[0,\infty]$ is a $k$-dimensional Lebesgue \textbf{u}$^\prime$-measure on $V$, then the Uniqueness Theorem implies $\mathcal{M}=\mathcal{M}^\prime$.
In other words, there is \emph{one and only one} $\sigma$-algebra that can serve as the domain of a $k$-dimensional Lebesgue \textbf{u}-measure on $V$ independent of the choice of the ordered orthonormal basis \textbf{u}.
This observation justifies the following terminology.
If the $\sigma$-algebra $\mathcal{M}$ on $V$ is the domain of a $k$-dimensional Lebesgue \textbf{u}-measure on V for some orthonormal basis \textbf{u} for $V$, then we call the elements of $\mathcal{M}$ the \emph{Lebesgue measurable subsets} of $V$ and we denote $\mathcal{M}$ by $\mathcal{M}(V)$.

The Uniqueness Theorem alone does not tell us that all Lebesgue measures on a finite dimensional inner product space are identical.  More precisely, the Uniqueness Theorem alone does not tell us that $\mu=\mu^\prime$ in the case that $\mu:\mathcal{M}(V)\rightarrow[0,\infty]$ is a $k$-dimensional Lebesgue \textbf{u}-measure on V and $\mu^\prime:\mathcal{M}(V)\rightarrow[0,\infty]$ is a $k$-dimensional Lebesgue \textbf{u}$^\prime$-measure on V where \textbf{u} and \textbf{u}$^\prime$ are \emph{distinct} ordered orthonormal bases for $V$.  
To reach that conclusion, we also need the following result.

\begin{theorem}
Suppose $V$ and $W$ are $k$-dimensional inner product spaces with ordered orthonormal bases \textbf{v} and \textbf{w}, respectively, $\lambda:\mathcal{M}(V)\rightarrow [0,\infty]$ is a $k$-dimensional Lebesgue \textbf{v}-measure on V and $\mu:\mathcal{M}(W)\rightarrow [0,\infty]$ is a $k$-dimensional Lebesgue \textbf{w}-measure on W.  If $f:V\rightarrow W$ is an isometry, then for every $A\in\mathcal{M}(V)$, $f(A)\in\mathcal{M}(W)$ and $\mu(f(A))=\lambda(A)$. 
\end{theorem}

A version of this result for isometries of $\mathbb{E}^n$ appears as Corollary 20.2.24 on page 435 of \cite{royden}.  The proof of this theorem has two parts.  First, the statement that an isometry sends elements of $\mathcal{M}(V)$ to elements of $\mathcal{M}(W)$ can be deduced from the facts that isometries satisfy a Lipschitz condition (defined on page 216 of \cite{royden}) and that a function from $V$ to $W$ which satisfies a Lipschitz condition sends elements of $\mathcal{M}(V)$ to elements of $\mathcal{M}(W)$.  (See Proposition 20.2.20 on page 432 of \cite{royden}.)  Second, the assertion that an isometry preserves measure ($\mu(f(A))=\lambda(A)$ for all $A\in\mathcal{M}(V)$) is proved in \cite{royden} in a way that relies on properties of the Lebesgue integral.  We will present a different proof here that avoids invoking the Lebesgue integral and instead depends on the Uniqueness Theorem 3.2.

\begin{proof}[Proof of Theorem 3.3]
First, consider the case in which $W=V$, $\textbf{w}=\textbf{v}$ and $f:V\rightarrow V$ is an isometry such that $f(0_V)=0_V$ where $0_V$ is the additive identity of $V$.
An elementary argument shows that any isometry $f:V\rightarrow W$ which sends $0_V$ to $0_W$ is linear.
(First observe that $||f(x)||_W=||x||_V$ for all $x\in V$.  
Next, show that $\langle f(x),f(y)\rangle_W=\langle x,y\rangle_V$ for all $x,y\in V$ by expanding both sides of the equation $(||f(x)-f(y)||_W)^2=(||x-y||_V)^2$.  
Third, show that $f(ax+by)-af(x)-bf(y)=0$ for all $x,y\in V$ and all $a,b\in\mathbb{R}$ by expanding $(||f(ax+by)-af(x)-bf(y)||_W)^2$.)  
Define the function $\lambda^\prime:\mathcal{M}(V)\rightarrow[0,\infty]$ by $\lambda^\prime(A)=\lambda(f(A))$.
Using that fact that $f$ is a linear function and a homeomorphism, one verifies that $\lambda^\prime$ is a non-zero translation invariant measure.  
(Proposition 2.1 is helpful in proving that $\lambda^\prime$ is translation invariant.)  
Hence, the Uniqueness Theorem provides a $\kappa\in(0,\infty)$ such that $\lambda^\prime=\kappa\lambda$.  
If $U=\{x\in V:||x||_V<k\}$, then $f(U)=U$ because f is an isometry that fixes $0_V$, and $\lambda(U)>0$ because $[\textbf{v}]\subset U$.  
Hence, $\kappa\lambda(U)=\lambda^\prime(U)= \lambda(f(U))=\lambda(U)$.  
Consequently, $\kappa=1$.  
Therefore, $\lambda^\prime=\lambda$.  
We conclude that  $\lambda(f(A))=\lambda(A)$ for all $A\in\mathcal{M}(V)$.

Second, consider the case in which $W=V,\textbf{w}=\textbf{v},\mu=\lambda$ and $f:V\rightarrow V$ is an isometry such that $f(0_V)=p\neq0_V$.  
Let $T_{-p}:V\rightarrow V$ denote the translation $T_{-p}(x)=x-p$.  
Let $A\in\mathcal{M}(V)$.  
Since $\lambda$ is translation invariant, then $\lambda(f(A))=\lambda(T_{-p}(f(A))$ and since $T_{-p}\circ f:V\rightarrow V$ is an isometry such that $T_{-p}\circ f(0_V)=0_V$, then $\lambda(T_{-p}\circ f(A))=\lambda(A)$.  
Hence, $\lambda(f(A))=\lambda(A)$.

Finally, consider the general case in which $W\neq V$.  
Suppose $\textbf{v}=(v_1,v_2,\cdots,v_k)$ and $\textbf{w}=(w_1,w_2,\cdots,w_k)$.  
Then a linear isometry $g:W\rightarrow V$ is determined by the equation $g(\sum_{1\leq i\leq k}t_iw_i)=\sum_{1\leq i\leq k}t_iv_i$.
Observe that $g([\textbf{w}])=[\textbf{v}]$.  
Define the function $\mu^\prime:\mathcal{M}(W)\rightarrow [0,\infty]$ by $\mu^\prime(B)=\lambda(g(B))$.  
As in the first paragraph of this proof, one can verify that $\mu^\prime$ is a non-zero translation invariant measure.  
Hence, the Uniqueness Theorem provides a $\kappa\in(0,\infty)$ such that $\mu^\prime=\kappa\mu$. 
Therefore, $\kappa=\kappa\mu([\textbf{w}])=\mu^\prime([\textbf{w}])=\lambda(g([\textbf{w}]))=\lambda([\textbf{v}])=1$.
Thus, $\mu^\prime=\mu$.  
Consequently, for $A\in\mathcal{M}(V)$, $\mu(f(A))=\mu^\prime(f(A))=\lambda(g\circ f(A))$.  
Since $g\circ f:V\rightarrow V$ is an isometry, then $\lambda(g\circ f(A))=\lambda(A)$. 
We conclude that $\mu(f(A))=\lambda(A)$.
\end{proof}

Theorem 3.3 allows us to settle the issue of the relationship between two $k$-dimensional Lebesgue measures arising from different ordered orthonormal bases for the same inner product space.

\begin{corollary}
If \textbf{u} and \textbf{u}$^\prime$ are ordered orthonormal bases for a $k$-dimensional inner product space $V$, $\mu:\mathcal{M}(V)\rightarrow[0,\infty]$ is a $k$-dimensional Lebesgue \textbf{u}-measure on $V$ and $\mu^\prime:\mathcal{M}(V)\rightarrow[0,\infty]$ is a $k$-dimensional Lebesgue $\textbf{u}^\prime$-measure on V, then $\mu=\mu^\prime$.
\end{corollary}

\begin{proof}
Since the identity function from $V$ to itself is an isometry, then Theorem 3.3 implies $\mu^\prime(A)=\mu(A)$ for every $A\in \mathcal{M}(V)$.
\end{proof}

From Corollary 3.4 we conclude that there is \emph{one and only one} $k$-dimensional Lebesgue \textbf{u}-measure on a $k$-dimensional inner product space $V$, independent of the choice of ordered orthonormal basis \textbf{u}.  We will subsequently call this measure \emph{$k$-dimensional Lebesgue measure on $V$} and denote it by $Vol_k:\mathcal{M}(V)\rightarrow[0,\infty]$.

Theorem 3.3 reveals that isometries between $k$-dimensional inner product spaces preserve Lebesgue measure.  We need one other result that describes how functions between $k$-dimensional inner product spaces affect Lebesgue measure.

\begin{theorem}
Suppose $V$ and $W$ are $k$-dimensional inner product spaces and $Vol_k:\mathcal{M}(V)\rightarrow[0,\infty]$ and $Vol_k:\mathcal{M}(W)\rightarrow[0,\infty]$ are the $k$-dimensional Lebesgue measures on $V$ and $W$, respectively. 
If $L:V\rightarrow W$ is a linear function, then for every $A\in\mathcal{M}(V)$, $L(A)\in\mathcal{M}(W)$.  
Furthermore, if \textbf{v} and \textbf{w} are ordered orthonormal bases for $V$ and $W$, respectively, then for every $A\in\mathcal{M}(V)$,
\[
Vol_k(L(A))=|Det(C_{\textbf{v},\textbf{w}}(L))|\;Vol_k(A).
\]
\end{theorem}

Since linear functions between finite dimensional inner product spaces satisfy a Lipschitz condition (Proposition 20.2.17 on page 430 of \cite{royden}) and since a function which satisfies a Lipschitz condition sends Lebesgue measurable sets to Lebesgue measurable sets (Proposition 20.2.18 on page 431 of \cite{royden}), then every linear function between finite dimensional inner product spaces $V$ and $W$ sends elements of $\mathcal{M}(V)$ to elements of $\mathcal{M}(W)$.  
A statement and proof of the equation $Vol_k(L(A))=|Det(C_{\textbf{v},\textbf{w}}(L))|\;Vol_k(A)$ for linear functions from $\mathbb{E}^n$ to itself appears as Corollary 20.2.23 on page 434 of \cite{royden}.  
(Again the proof of this result in \cite{royden} relies on properties of the Lebesgue integral.  
A simple proof which avoids the Lebesgue integral can be given using a linear algebra theorem known as the \emph{Polar Decomposition Theorem}.  
For linear functions between finite dimensional inner product spaces, this theorem has a proof which, though simple, is beyond the scope of this article.

Suppose $\mathcal{M}$ is a $\sigma$-algebra on a set $X$, $\mathcal{N}$ is a $\sigma$-algebra on a set $Y$ and $f:X\rightarrow Y$ is a function such that $f(A)\in\mathcal{N}$ for all $A\in\mathcal{M}$.  
Then $f$ determines a function from $\mathcal{M}$ to $\mathcal{N}$ which we denote $f_{\ast}:\mathcal{M}\rightarrow\mathcal{N}$ and specify by the equation $f_{\ast}(A)=f(A)$ for all $A\in\mathcal{M}$.

Suppose $(X,\rho)$ and $(Y,\sigma)$ are pseudometric spaces and $f:X\rightarrow Y$ is a function.
$f$ is an \emph{isometry} if $f(X)=Y$ and $\sigma(f(x),f(x^\prime))=\rho(x,x^\prime)$ for all $x,x^\prime\in X$.
$f$ is a \emph{dilation with dilation constant $\delta$} if $\delta\in[0,\infty)$ and $\sigma(f(x),f(x^\prime))=\delta\rho(x,x^\prime)$ for all $x,x^\prime\in X$.  
Note that isometries and dilations are continuous.

Observe that if $f:X\rightarrow Y$ is a function and $A$ and $B$ are subsets of $X$, then $f(A)\Delta f(B)\subset f(A\Delta B)$, with equality holding if $f$ is injective.  Theorems 3.3 and 3.5 together with this observation yield the following conclusions.

\begin{corollary}
Suppose $V$ and $W$ are $k$-dimensional inner product spaces, $Vol_k:\mathcal{M}(V)\rightarrow[0,\infty]$ and $Vol_k:\mathcal{M}(W)\rightarrow[0,\infty]$ are the $k$-dimensional Lebesgue measures on $V$ and $W$, respectively, and $\mathcal{M}(V)$ and $\mathcal{M}(W)$ are assigned the symmetric difference 
pseudometrics associated with these Lebesgue measures.  

\medskip\noindent\textit{i})\;If $f:V\rightarrow W$ is an isometry, then $f_{\ast}:\mathcal{M}(V)\rightarrow \mathcal{M}(W)$ is also an isometry.

\medskip\noindent\textit{ii})\;If $L:V\rightarrow W$ is a linear function and \textbf{v} and \textbf{w} are ordered orthonormal bases for $V$ and $W$, respectively, then 
$L_{\ast}:\mathcal{M}(V)\rightarrow \mathcal{M}(W)$ is a dilation with dilation constant $|Det(C_{\textbf{v},\textbf{w}}(L)|$.  \qed
\end{corollary}

Next we provide a technique for calculating $k$-dimensional Lebesgue measure in a $k$-dimensional inner product space $V$. 
Suppose \textbf{u} = $(u_1,u_2,\cdots,u_k)$ is an ordered orthonormal basis for $V$.  
For real numbers $b_i<b_i^\prime$ for $1\leq i\leq k$, we call any set of the form
\[
B=\Biggl\{\sum_{1\leq i\leq k}t_iu_i:b_i\leq t_i\leq b_i^\prime\text{ for }1\leq i\leq k\Biggr\}
\]
a \emph{\textbf{u}-box} and we define the \emph{\textbf{u}-volume} of $B$ to be
\[
vol_{\textbf{u}}(B)= \prod_{1\leq i\leq k}(b_i^\prime -b_i).
\]
Also let
\[
int(B)=\Biggl\{\sum_{1\leq i\leq k}t_iu_i:b_i<t_i<b_i^\prime\text{ for }1\leq i\leq k\Biggr\}
\]
and note that $int(B)$ in non-empty.
Let $\mathbb{B}(V,\textbf{u})$ denote the set of all countable collections \textbf{u}-boxes in $V$. 
If $\mathcal{B}\in\mathbb{B}(V,\textbf{u})$, we define the \emph{\textbf{u}-volume sum of} $\mathcal{B}$ to be
\[
vol_{\textbf{u}}sum(\mathcal{B})=\sum_{B\in\mathcal{B}}vol_{\textbf{u}}(B).
\]		
We stipulate that $\mathbb{B}(V,\textbf{u})$ contains the empty collection $\emptyset$ and that $vol_{\textbf{u}}sum(\emptyset)=0$. 
Observe that since $\mathcal{M}(V)$ is a Borel $\sigma$-algebra and since all \textbf{u}-boxes are closed subsets of $V$, then all \textbf{u}-boxes are elements of $\mathcal{M}(V)$ and $\bigcup\mathcal{B}\in\mathcal{M}(V)$ for all $\mathcal{B}\in
\mathbb{B}(V,\textbf{u})$.
The connection between countable collections of \textbf{u}-boxes in $V$ and $k$-dimensional Lebesgue measure on $V$ is the following result.

\begin{theorem}
If \textbf{u} is an ordered orthonormal basis for a $k$-dimensional inner product space $V$, then
\begin{center}
    $Vol_k(A)=inf\{vol_{\textbf{u}}sum(\mathcal{B}):\mathcal{B}\in\mathbb{B}(V,\textbf{u})\text{ and }A\subset\bigcup\mathcal{B}\}$.
\end{center}
\end{theorem}

The usual approach to constructing Lebesgue measure on a finite dimensional inner product space $V$ is to define a function $\mu^\ast:\mathcal{P}(V)\rightarrow[0,\infty]$ by the formula
\begin{center}
$\mu^\ast(A)=inf\{vol_{\textbf{u}}sum(\mathcal{B}):\mathcal{B}\in\mathbb{B}(V,\textbf{u})\text{ and }A\subset\bigcup\mathcal{B}\}$.
\end{center}
Here $\mathcal{P}(V)=\{A:A\subset V\}$ is the \emph{power set} of $V$.
The function $\mu^\ast$ is called \emph{Lebesgue \textbf{u}-outer measure on $V$}.  
Then a subcollection $\mathcal{M}(V,\textbf{u})$ of $\mathcal{P}(V)$ is specified by the formula 
\[
\mathcal{M}(V,\textbf{u})=\{A\in\mathcal{P}(V):\mu^\ast(B\cap A)+\mu^\ast(B-A)=\mu^\ast(B)\text{ for all }B\in\mathcal{P}(V)\}. 
\]
The elements of $\mathcal{M}(V,\textbf{u})$ are called \emph{Lebesgue \textbf{u}-measurable subsets of $V$}. 
Next Carath\'{e}odory’s Theorem (Theorem 17.4.8 on page 34 of \cite{royden}) is invoked to reveal that $\mathcal{M}(V,\textbf{u})$ is a $\sigma$-algebra on $V$ and the restriction 
$\mu^\ast|\mathcal{M}(V,\textbf{u})$ is a complete measure on $V$.
Finally, one shows that $\mu^\ast|\mathcal{M}(V,\textbf{u})$ is Borel, regular, translation invariant  $\mu^\ast([\textbf{u}])=1$; in other words, $\mu^\ast|\mathcal{M}(V,\textbf{u})$ is a $k$-dimensional Lebesgue \textbf{u}-measure on $V$.  
(See Theorems 20.2.11 and 20.2.13 and Problem 20.2.20 on pages 426-427, 427-428 and 435, respectively, of \cite{royden}.)  
Then the Uniqueness Theorem 3.2 and Corollary 3.4 imply $\mathcal{M}(V,\textbf{u})=\mathcal{M}(V)$ and $\mu^\ast|\mathcal{M}(V,\textbf{u})=Vol_k$, thereby yielding Theorem 3.7.  
This approach to the construction of Lebesgue measure was pioneered in \cite{caratheodory}.  In addition to \cite{royden}, we mention \cite{halmos} as a historically influential comprehensive source for information about measure theory.

We again consider a $k$-dimensional inner product space $V$ and let \textbf{u} be an ordered orthonormal basis for $V$.  
We introduce a special subcollection of $\mathbb{B}(V,\textbf{u})$ that is quite useful in both the development of Lebesgue measure and in the remainder of this paper.  
A countable collection $\mathcal{B}$ of \textbf{u}-boxes in $V$ is \emph{almost disjoint} if $\{int(B):B\in\mathcal{B}\}$ is
pairwise disjoint.  
Let 
\[
\mathbb{B}_0(V,\textbf{u})=\{\mathcal{B}\in\mathbb{B}(V,\textbf{u}):\mathcal{B}\text{ is almost disjoint}\}.
\]
The following lemmas state two basic properties of almost disjoint collections of \textbf{u}-boxes.

\begin{lemma}
Suppose $V$ is a $k$-dimensional inner product space and \textbf{u} is an ordered orthonormal basis for $V$.  
If $\mathcal{B}\in \mathbb{B}_0(V,\textbf{u})$, then $Vol_k(\bigcup\mathcal{B})=vol_\textbf{u}sum(\mathcal{B})=\sum_{B\in\mathcal{B}}Vol_k(B)$.
\end{lemma}

\begin{lemma}
Suppose $V$ is a $k$-dimensional inner product space and \textbf{u} is an ordered orthonormal basis for $V$.
For every open subset $U$ of $V$, there is a $\mathcal{B}\in \mathbb{B}_0(V,\textbf{u})$ such that $\bigcup\mathcal{B}=U$.
\end{lemma}

These results appear in \cite{tao} as Lemmas 1.2.9 and 1.2.11 on pages 23-25.

The next result hints at the value of the concept of almost disjoint collections of boxes.  
This theorem plays a key role in our proof of the Product Formula for Volume.

\begin{theorem}
Suppose $V$ is a $k$-dimensional inner product space and \textbf{u} is an ordered orthonormal basis for $V$.
Let $\rho$ be the symmetric difference pseudometric on $\mathcal{M}(V)$ associated with $k$-dimensional Lebesgue measure $Vol_k$. 
Then $\{\bigcup\mathcal{B}:\mathcal{B}\in\mathbb{B}_0(V,\textbf{u})\}$ is a dense subset of $\mathcal{M}(V)$ in the topology on $\mathcal{M}(V)$ determined by $\rho$.
\end{theorem}

\begin{proof}
Lemma 3.9 tells us that it suffices to prove that the set of all open subsets of $V$ is dense in $\mathcal{M}(V)$. 
To this end, let $A\in\mathcal{M}(V)$ and let $\epsilon>0$.
Since $Vol_k$ is $\sigma$-finite, there is a countable subset
$\{A_i:i\geq1\}$ of $\mathcal{M}(V)$ such that $A=\bigcup_{i\geq1}A_i$ and $Vol_k(A_i)<\infty$ for each $i\geq1$.
Since $Vol_k$ is regular, then for each $i\geq1$, there is an open subset $U_i$ of $V$ such that $A_i\subset U_i$ and $Vol_k(U_i)<Vol_k(A_i)+\epsilon/2^i$. 
Hence, for each $i\geq1$, $Vol_k(U_i-A_i)=Vol_k(U_i)-Vol_k(A_i)<\epsilon/2^i$.
Let $U=\bigcup_{i\geq1}U_i$.
Then $U$ is an open subset of $V$ such that $A\subset U$.
Furthermore, $U-A=\bigcup_{i\geq1}(U_i-A)\subset\bigcup_{i\geq1}(U_i-A_i)$.  
Hence,

\medskip\hspace{24mm}$\rho(A,U)=Vol_k(U-A)\leq\sum_{i\geq1}Vol_k(U_i-A_i)<\epsilon$.
\end{proof}

Suppose $V_0$ is a $k$-dimensional vector subspace of $\mathbb{E}^n$.  
Clearly, the dot product on $\mathbb{E}^n$ restricts to an inner product on $V_0$.  
Hence, there is an associated unique Borel $\sigma$-algebra $\mathcal{M}(V_0)$ of Lebesgue measurable subsets of $V_0$ and there is a unique $k$-dimensional Lebesgue measure $Vol_k:\mathcal{M}(V_0)\rightarrow[0,\infty]$ on $V_0$.  
Now consider a $k$-dimensional affine subspace $V$ of $\mathbb{E}^n$ such that $V=T_p(V_0)$ where $p\in\mathbb{E}^n$.
If we restrict the Euclidean metric on $\mathbb{E}^n$ (i.e., the metric on $\mathbb{E}^n$ associated with the dot product) to $V$, then $T_p|V_0:V_0\rightarrow V$ becomes an isometry.
Hence, we feel justified in calling the set $\{T_p(A):A\in\mathcal{M}(V_0)\}$ the \emph{$\sigma$-algebra of Lebesgue measurable subsets of $V$}, denoting this set by $\mathcal{M}(V)$, and in defining \emph{$k$-dimensional Lebesgue measure} $Vol_k:\mathcal{M}(V)\rightarrow[0,\infty]$ on $V$ by $Vol_k(T_p(A))=Vol_k(A)$ for $A\in\mathcal{M}(V_0)$.  
In further support of these choices, we observe that the definitions of $\mathcal{M}(V)$ and $Vol_k:\mathcal{M}(V)\rightarrow[0,\infty]$ are independent of the choice of the point $p\in\mathbb{E}^n$ such that $T_p(V_0)=V$.  
For suppose $p$ and $q\in\mathbb{E}^n$ such that $T_p(V_0)=T_q(V_0)=V$.  
Then $q-p\in V_0$.  
Therefore, $T_{q-p}$ restricts to a translation of $V_0$ which sends $\mathcal{M}(V)$ onto itself.  
Hence, $\{T_q(A):A\in\mathcal{M}(V_0)\}=\{T_p\circ T_{q-p}(A):A\in\mathcal{M}(V_0)\}=\{T_p(A):A\in\mathcal{M}(V_0)\}$.  
It follows that the definition of $\mathcal{M}(V)$ is independent of the choice of $p$.  
Also if $A,B\in\mathcal{M}(V_0)$ such that $T_p(A)=T_q(B)$, then $T_{q-p}(B)=A$.  
The translation invariance of $Vol_k:\mathcal{M}(V_0)\rightarrow[0,\infty]$ then implies $Vol_k(B)=Vol_k(A)$.  
Consequently, the definition of $Vol_k:\mathcal{M}(V)\rightarrow[0,\infty]$ is independent of the choice of $p$.

We continue to extend our terminology concerning Lebesgue measure to affine subspaces of $\mathbb{E}^n$.  
Again, suppose $V$ is a $k$-dimensional affine subspace of $\mathbb{E}^n$.  
Say $V=T_p(V_0)$ where $V_0$ is a $k$-dimensional vector subspace of $\mathbb{E}^n$ and $p\in\mathbb{E}^n$.  
Echoing our previous definitions, if \textbf{u} is an ordered orthonormal basis for $V_0$, we call a subset of $V$ a \emph{\textbf{u}-box in $V$} if it is of the form $T_p(B)$ where $B$ is a \textbf{u}-box in $V_0$.  
Let $\mathbb{B}(V,\textbf{u})$ denote the set of all countable collections of \textbf{u}-boxes in $V$.  
As above, we say that a collection $\mathcal{B}\in\mathbb{B}(V,\textbf{u})$ is \emph{almost disjoint} if $\{int_V(B):B\in\mathcal{B}\}$ is pairwise disjoint (where $int_V$ denotes the interior of a subset of $V$ in the subspace topology on $V$ that it inherits as a subset of $\mathbb{E}^n$).  
We then let $\mathbb{B}_0(V,\textbf{u})=\{\mathcal{B}\in\mathbb{B}(V,\textbf{u}):\mathcal{B}\text{ is almost disjoint}\}$. 
It follows that the properties of Lebesgue measure asserted for $V_0$ by Corollary 3.6, Lemmas 3.8 and 3.9 and Theorem 3.10 effortlessly transfer to properties of $V$.  
In particular, in one of the subsequent arguments, we will need:

\begin{corollary3.10prime}
Suppose $V$ is a $k$-dimensional affine subspace of  $\mathbb{E}^n$, $V=T_p(V_0)$ where $V_0$ is a $k$-dimensional vector subspace of $\mathbb{E}^n$ and $p\in\mathbb{E}^n$, and \textbf{u} is an ordered orthonormal basis for $V_0$.  
Then $\{\bigcup\mathcal{B}:\mathcal{B}\in\mathbb{B}_0(V,\textbf{u})\}$ is a dense subset of $\mathcal{M}(V)$ in the topology on $\mathcal{M}(V)$ determined by the symmetric difference pseudometric associated with $Vol_k$.  \qed
\end{corollary3.10prime}

We close this section with a result that will be used later.

\begin{proposition}
Suppose $V$ is a $k$-dimensional affine subspace of $\mathbb{E}^n$, $W$ is a $k$-dimensional inner product space and $L:\mathbb{E}^n\rightarrow W$ is a linear function.
Let $V_0$ be the $k$-dimensional vector subspace of $\mathbb{E}^n$ such that $V=T_p(V_0)$ where $p\in\mathbb{E}^n$, let \textbf{u} be an ordered orthonormal basis for $V_0$, and let $Vol_k:\mathcal{M}(W)\rightarrow[0,\infty]$ denote $k$-dimensional Lebesgue measure on $W$.
Then for every $\mathcal{B}\in\mathbb{B}_0(V,\textbf{u})$, $Vol_k(L(\bigcup\mathcal{B})=\sum_{B\in\mathcal{B}}Vol_k(L(B))$.
\end{proposition}

\begin{proof}
First consider the special case $V=V_0$.
Then $\mathcal{B}\in\mathbb{B}_0(V_0,\textbf{u})$.
Let \textbf{w} be an ordered orthonormal basis for $W$.
Then Theorem 3.5 and Lemma 3.7 imply
\begin{center}
$Vol_k(L(\bigcup\mathcal{B}))=|Det(C_{\textbf{u},\textbf{w}}(L))|\;Vol_k(\bigcup\mathcal{B})=|Det(C_{\textbf{u},\textbf{w}}(L))|\;\sum_{B\in\mathcal{B}}Vol_k(B)$
\end{center}
\begin{center}
    $=\sum_{B\in\mathcal{B}}|Det(C_{\textbf{u},\textbf{w}}(L))|\;Vol_k(B)=\sum_{B\in\mathcal{B}}Vol_k(L(B))$.
\end{center}

Now consider the general case $V=T_p(V_0)$.
Given $\mathcal{B}\in\mathbb{B}_0(V,\textbf{u})$, there is a $\mathcal{B}_0\in\mathbb{B}_0(V_0,\textbf{u})$ such that $\mathcal{B}=\{T_p(B):B\in\mathcal{B}_0\}$.  
Since $T_p|V_0:V_0\rightarrow V$ is a homeomorphism, then $\mathcal{B}_0\in\mathbb{B}_0(V_0,\textbf{u})$.  
Hence, the argument in the previous paragraph shows that
\begin{center}
$Vol_k(L(\bigcup\mathcal{B}_0))=\sum_{B\in\mathcal{B}_0}Vol_k(L(B))$.
\end{center}
Note that $T_p(\bigcup\mathcal{B}_0)=\bigcup\mathcal{B}$.
Using Proposition 2.1 and the fact that $Vol_k:\mathcal{M}(W)\rightarrow[0,\infty]$ is translation invariant, we have

\bigskip\hspace{10mm}$Vol_k(L(\bigcup\mathcal{B}))=Vol_k(L\circ T_p(\bigcup\mathcal{B}_0))=Vol_k(T_{L(p)}\circ L(\bigcup\mathcal{B}_0))=$

\bigskip\hspace{10mm}$Vol_k(L(\bigcup\mathcal{B}_0))=\sum_{B\in\mathcal{B}_0}Vol_k(L(B))=\sum_{B\in\mathcal{B}_0}Vol_k(T_{L(p)}\circ L(B))=$

\bigskip\hspace{28mm}$\sum_{B\in\mathcal{B}_0}Vol_k(L\circ T_p(B))=\sum_{B\in\mathcal{B}}Vol_k(L(B))$.
\end{proof}
\smallskip

\section{The Proof of the Product Formula for Volume - for parallelopipeds}
Let $V$ be a $k$-dimensional vector subspace of $\mathbb{E}^n$ and let $\textbf{u}=(u_1,u_2,\cdots,u_k)$ be an ordered orthonormal basis for $V$.  If $\textbf{x}=(x_1,x_2,\cdots,x_k)\in V^k$, let $[\textbf{x}]=[x_1,x_2,\cdots,x_k]$ denote the parallelopiped 
		$\{\sum_{1\leq i\leq k}t_ix_i:0\leq t_i\leq 1\text{ for }1\leq i\leq k\}$.  
Then $Vol_k([\textbf{u}])=1$ because $Vol_k$ coincides with $k$-dimensional Lebesgue \textbf{u}-measure on $V$ by Corollary 3.4.

Suppose $\textbf{x}=(x_1,x_2,\cdots,x_k)\in V^k$ and define the linear map $L:V\rightarrow V$ by the equations $L(u_i)=x_i$ for $1\leq i\leq k$.  
Then $L(\textbf{u})=(L(u_1),L(u_2),\cdots,L(u_k))=(x_1,x_2,\cdots,x_k)=\textbf{x}$ and $L([\textbf{u}])=[\textbf{x}]$.  
Proposition 2.3 implies $C_{\textbf{u},\textbf{u}}(L)=C_{\textbf{u}}(L(\textbf{u}))\\
=C_{\textbf{u}}(\textbf{x})$. Hence, Theorem 3.5 implies

\medskip$(\alpha)\hspace{14mm} Vol_k([\textbf{x}])=
\vert Det(C_{\textbf{u},\textbf{u}}(L))\vert
Vol_k([\textbf{u}])=
\vert Det(C_{\textbf{u}}(\textbf{x})\vert$.

\medskip\noindent Now suppose $\textbf{x}=(x_1,x_2,\cdots,x_k)$ and $\textbf{y}=(y_1,y_2,\cdots,y_k)\in V^k$.  
We will prove

\medskip$(\beta)\hspace{7mm} \langle\langle x_1\wedge x_2\wedge\cdots\wedge x_k,y_1\wedge y_2\wedge\cdots\wedge y_k\rangle\rangle=Vol_k([\textbf{x}])Vol_k([\textbf{y}])$.

\medskip\noindent For the purpose of proving the Product Formula for Volume for $A=[\textbf{x}]$ and $B=[\textbf{y}]$, the specific orders of the entries of \textbf{x} and \textbf{y} are immaterial.  
So let us interchange $y_1$ and $y_2$ if necessary so that $Det(C_{\textbf{u}}(\textbf{x}))$ and $Det(C_{\textbf{u}}(\textbf{y}))$ have the same sign.  
The definition of $\langle\langle\;,\;\rangle\rangle$ implies

\begin{center}
$\langle\langle x_1\wedge x_2\wedge\cdots\wedge x_k,y_1\wedge y_2\wedge\cdots\wedge y_k\rangle\rangle=Det(x_i\sbullet y_i)$
\end{center}

\noindent Let $(a_{i,j})=C_{\textbf{u}}(\textbf{x})$ and $(b_{i,j})=C_{\textbf{u}}(\textbf{y})$.
Then $x_i=\sum_{1\leq r\leq k}a_{i,r}u_r$ and $y_i=\sum_{1\leq s\leq k}b_{i,s}u_s$.
Therefore,

\begin{center}
$x_i\sbullet y_j=(\sum_ra_{i,r}u_r)\sbullet(\sum_sb_{j,s}u_s)=\sum_{r,s}a_{i,r}b_{j,s}(u_r\sbullet u_s)$
\end{center}

\begin{center}
$=\sum_ra_{i,r}b_{j,r}=\sum_ra_{i,r}(b_{r,j}^T)$.
\end{center}

\noindent Thus, $(x_i\sbullet y_j)=C_{\textbf{u}}(\textbf{x})(C_{\textbf{u}}(\textbf{y})^T)$.
Hence,

\medskip$(\gamma)\hspace{3.5mm} \langle\langle x_1\wedge x_2\wedge\cdots\wedge x_k,y_1\wedge y_2\wedge\cdots\wedge y_k\rangle\rangle=Det(C_{\textbf{u}}(\textbf{x}))Det(C_{\textbf{u}}(\textbf{y}))$.

\medskip\noindent Since $Det(C_{\textbf{u}}(\textbf{x}))$ and $Det(C_{\textbf{u}}(\textbf{y}))$ have the same sign, then we have:

\medskip$(\delta)\hspace{1.25mm} \langle\langle x_1\wedge x_2\wedge\cdots\wedge x_k,y_1\wedge y_2\wedge\cdots\wedge y_k\rangle\rangle=|Det(C_{\textbf{u}}(\textbf{x}))||Det(C_{\textbf{u}}(\textbf{y}))|$.

\medskip\noindent $(\delta)$ and $(\alpha)$ clearly imply $(\beta)$.

Again suppose $\textbf{x}=(x_1,x_2,\cdots,x_k)$ and $\textbf{y}=(y_1,y_2,\cdots,y_k)\in V^k$ with entries ordered so that $Det(C_{\textbf{u}}(\textbf{x}))$ and $Det(C_{\textbf{u}}(\textbf{y}))$ have the same sign.
Let $J\in\mathcal{S}(n,k)$.
Recall that $\mathbb{E}^J=\{(z_1,z_2,\cdots,z_n)\in\mathbb{E}^n:z_i=0\text{ if }i\notin J\}$ and $\pi_J:\mathbb{E}^n\rightarrow\mathbb{E}^n$ is defined by $\pi_J(z_1,z_2,\cdots,z_n)=(w_1,w_2,\cdots,w_n)$ where $w_i=z_i$ if $i\in J$ and $w_i=0$ if $i\notin J$.
Let $\textbf{e}=(e_1,e_2,\cdots,e_n)$ be the standard ordered orthonormal basis for $\mathbb{E}^n$.
Write $J=\{j_1,j_2,\cdots,j_k\}$ where $1\leq j_1<j_2<\cdots<j_k\leq n$, let $\textbf{e}_J=(e_{j_1},e_{j_2},\cdots,e_{j_k})$ and let $\wedge_J\textbf{e}=\textit{i}(\textbf{e}_J)=e_{j_1}\wedge e_{j_2}\wedge\cdots\wedge e_{j_k}$.
We now calculate $\langle\langle x_1\wedge x_2\wedge\cdots\wedge x_k,\wedge_J\textbf{e}\rangle\rangle$.
First observe that for $1\leq i\leq k$ and $1\leq t\leq k$, $x_i\sbullet e_{j_t}=\pi_J(x_i)\sbullet e_{j_t}$.
Hence, the definition of $\langle\langle\;,\;\rangle\rangle$ implies
\[
\langle\langle x_1\wedge x_2\wedge\cdots\wedge x_k,\wedge_J\textbf{e}\rangle\rangle=Det(x_j\sbullet e_{j_t})=Det(\pi_J(x_i)\sbullet e_{j_t})
\]
 \[
=\langle\langle\pi_J(x_1)\wedge\pi_J(x_2)\wedge\cdots\wedge\pi_J(x_k),\wedge_J\textbf{e}\rangle\rangle.
\]
Let $\pi_J(\textbf{x})=(\pi_J(x_1),\pi_J(x_2),\cdots,\pi_J(x_k))$.
Since $\mathbb{E}^J$ is a $k$-dimensional vector subspace of $\mathbb{E}^n$, and both $\pi_J(\textbf{x})$ and $\textbf{e}_j\in(\mathbb{E}^J)^k$, then we can apply $(\gamma)$ to conclude that
\[
\langle\langle x_1\wedge x_2\wedge\cdots\wedge x_k,\wedge_J\textbf{e}\rangle\rangle=Det(C_{\textbf{e}_J}(\pi_J(\textbf{x})))Det(C_{\textbf{e}_J}(\textbf{e}_J))=Det(C_{\textbf{e}_J}(\pi_J(\textbf{x}))).
\]
Similarly
\[
\langle\langle y_1\wedge y_2\wedge\cdots\wedge y_k,\wedge_J\textbf{e}\rangle\rangle=Det(C_{\textbf{e}_J}(\pi_J(\textbf{y}))).
\]
Hence,

\medskip$(\epsilon)\hspace{15.3mm}\langle\langle x_1\wedge x_2\wedge\cdots\wedge x_k,\wedge_J\textbf{e}\rangle\rangle\langle\langle y_1\wedge y_2\wedge\cdots\wedge y_k,\wedge_J\textbf{e}\rangle\rangle$
\[
=Det(C_{\textbf{e}_J}(\pi_J(\textbf{x})))Det(C_{\textbf{e}_J}(\pi_J(\textbf{y}))).
\]
We assert that $Det(C_{\textbf{e}_J}(\pi_J(\textbf{x})))Det(C_{\textbf{e}_J}(\pi_J(\textbf{y})))\geq0$.
Indeed, since
\[
C_{\textbf{e}_J}(\pi_J(\textbf{x}))=C_\textbf{u}(\textbf{x})C_{\textbf{u},\textbf{e}_J}(\pi_J)\text{ and }C_{\textbf{e}_J}(\pi_J(\textbf{y}))=C_\textbf{u}(\textbf{y})C_{\textbf{u},\textbf{e}_J}(\pi_J), 
\]
then
\[
Det(C_{\textbf{e}_J}(\pi_J(\textbf{x})))Det(C_{\textbf{e}_J}(\pi_J(\textbf{y})))=Det(C_\textbf{u}(\textbf{x}))Det(C_\textbf{u}(\textbf{y}))(Det(C_{\textbf{u},\textbf{e}_J}(\pi_J)))^2.
\]
Since $Det(C_\textbf{u}(\textbf{x}))$ and $Det(C_\textbf{u}(\textbf{y}))$ have the same sign, our assertion follows.
Hence, $(\epsilon)$ implies
\[
\langle\langle x_1\wedge x_2\wedge\cdots\wedge x_k,\wedge_J\textbf{e}\rangle\rangle\langle\langle y_1\wedge y_2\wedge\cdots\wedge y_k,\wedge_J\textbf{e}\rangle\rangle
\]
\[
=|Det(C_{\textbf{e}_J}(\pi_J(\textbf{x})))|\;|Det(C_{\textbf{e}_J}(\pi_J(\textbf{y})))|.
\]
Since $[\pi_J(\textbf{x})]=\pi_J([\textbf{x}])$ and $[\pi_J(\textbf{y})]=\pi_J([\textbf{y}])$, then $(\alpha)$ implies 
\[
|Det(C_{\textbf{e}_J}(\pi_J(\textbf{x})))|=Vol_k(\pi_J([\textbf{x}]))\text{ and }|Det(C_{\textbf{e}_J}(\pi_J(\textbf{y})))|=Vol_k(\pi_J([\textbf{y}])).
\]
It follows that

\medskip$(\zeta)\hspace{15.3mm}\langle\langle x_1\wedge x_2\wedge\cdots\wedge x_k,\wedge_J\textbf{e}\rangle\rangle\langle\langle y_1\wedge y_2\wedge\cdots\wedge y_k,\wedge_J\textbf{e}\rangle\rangle$
\[
=Vol_k(\pi_J([\textbf{x}]))Vol_k(\pi_J([\textbf{y}])).
\]
Because $\{\wedge_J\textbf{e}:J\in\mathcal{S}(n,k)\}$ is an orthonormal basis for $\Lambda_k(\mathbb{E}^n)$, the Product Formula for Inner Products when applied to $\Lambda_k(\mathbb{E}^n)$ with its inner product $\langle\langle\;,\;\rangle\rangle$ tells us:
\[
\langle\langle x_1\wedge x_2\wedge\cdots\wedge x_k,y_1\wedge y_2\wedge\cdots\wedge y_k\rangle\rangle=
\]
\begin{center}
$\sum_{J\in\mathcal{S}(n,k)}\langle\langle x_1\wedge x_2\wedge\cdots\wedge x_k,\wedge_J\textbf{e}\rangle\rangle\langle\langle y_1\wedge y_2\wedge\cdots\wedge y_k,\wedge_J\textbf{e}\rangle\rangle$.
\end{center}
Combining this equation with $(\beta)$ and $(\zeta)$ yields the Product Formula for 

\noindent$k$-dimensional parallelopipeds that lie in a $k$-dimensional vector subspace of $\mathbb{E}^n$:

\medskip$(\eta)\hspace{8.5mm}Vol_k([\textbf{x}])Vol_k([\textbf{y}])=\sum_{J\in\mathcal{S}(n,k)}Vol_k(\pi_J([\textbf{x}]))Vol_k(\pi_j([\textbf{y}])).$  \qed
\medskip

\section{The Proof of the Product Formula for Volume - for Lebesgue measurable subsets of parallel $k$-dimensional affine subspaces of $\mathbb{E}^n$}
Let $V$ and $W$ be identical or parallel $k$-dimensional affine subspaces of $\mathbb{E}^n$. 
We first prove the Product Formula for Volume for boxes in V and W.  
There is a $k$-dimensional vector subspace $U$ of $\mathbb{E}^n$ and points $p,q\in\mathbb{E}^n$ such that $V=T_p(U)$ and $W=T_q(U)$.
Let $\textbf{u}=(u_1,u_2,\cdots,u_k)$ be an ordered orthonormal basis for $U$.
Suppose $B$ is a \textbf{u}-box in $V$ and $C$ is a \textbf{u}-box in $W$.
Then $B=T_p(B_0)$ and $C=T_q(C_0)$ where $B_0$ and $C_0$ are \textbf{u}-boxes in $U$.
Therefore, $Vol_k(B)=Vol_k(B_0)$ and $Vol_k(C)=Vol_k(C_0)$.
Suppose
\begin{center}
$B_0=\{\sum_{1\leq i\leq k}t_iu_i:b_i\leq t_i\leq b_i^\prime\text{ for }1\leq i\leq k\}$ 
\end{center}
and
\begin{center}
$C_0=\{\sum_{1\leq i\leq k}t_iu_i:c_i\leq t_i\leq c_i^\prime\text{ for }1\leq i\leq k\}$ 
\end{center}
where $b_i<b_i^\prime$ and $c_i<c_i^\prime$ are real numbers for $1\leq i\leq k$.
Let $\textbf{x}=(x_1,x_2,\cdots,x_k)$ and let $\textbf{y}=(y_1,y_2,\cdots,y_k)$ where $x_i=b_i^\prime-b_i$ and $y_i=c_i^\prime-c_i$ for $1\leq i\leq k$.
Also let $p^\prime=\sum_{1\leq i\leq k}b_iu_i$ and $q^\prime=\sum_{1\leq i\leq k}c_iu_i$.
Then $B_0=T_{p^\prime}([\textbf{x}])$ and $C_0=T_{q^\prime}([\textbf{y}])$.
Since $p^\prime$ and $q^\prime\in U$, then $T_{p^\prime}$ and $T_{q^\prime}$ map $U$ onto itself.
Since $Vol_k$ is translation invariant on $U$, it follows that $Vol_k(B_0)=Vol_k([\textbf{x}])$ and $Vol_k(C_0)=Vol_k([\textbf{y}])$.
We conclude that

\medskip$(\theta)\hspace{24.8mm}Vol_k(B)Vol_k(C)=Vol_k([\textbf{x}])Vol_k([\textbf{y}]).$

\medskip Let $J\in\mathcal{S}(n,k)$.
For $z\in\mathbb{E}^n$, since $\pi_J(z)\in\mathbb{E}^J$, then $T_{\pi_J(z)}$ maps $\mathbb{E}^J$ to itself. 
Since $Vol_k$ is translation invariant on $\mathbb{E}^J$, then we have $Vol_k(T_{\pi_J(z)}(A))=Vol_k(A)$ for every Lebesgue measurable subset $A$ of $\mathbb{E}^J$.
Recall that for $z,w\in\mathbb{E}^n$, $T_z\circ T_w=T_{z+w}$ and, by Proposition 
2.1 applied to the linear map $\pi_J:\mathbb{E}^n\rightarrow\mathbb{E}^J$, $\pi_J\circ T_z=T_{\pi_J(z)}\circ \pi_J$.  
Hence,
\[
Vol_k(\pi_J(B))=Vol_k(\pi_J(T_p(B_0)))=Vol_k(\pi_J(T_p(T_{p^\prime}([\textbf{x}]))))=  
\]
\[
Vol_k(\pi_J\circ T_{p+p^\prime}([\textbf{x}]))=Vol_k(T_{\pi_J(p+p^\prime)}(\pi_J([\textbf{x}]))=Vol_k(\pi_J([\textbf{x}])).
\]
Similarly, $Vol_k(\pi_J(C))=Vol_k(\pi_J([\textbf{y}]))$.
Therefore, 

\medskip$(\kappa)\hspace{11.7mm}Vol_k(\pi_J(B))Vol_k(\pi_J(C))=Vol_k(\pi_J([\textbf{x}]))Vol_k(\pi_J([\textbf{y}]))$ 
\begin{center}
    for every $J\in\mathcal{S}(n,k)$.
\end{center}

Combining $(\eta)$, $(\theta)$ and $(\kappa)$ yields a Product Formula for Volume that holds for \textbf{u}-boxes $B$ and $C$ in parallel $k$-dimensional affine subspaces $V$ and $W$ of $\mathbb{E}^n$:
\begin{center}
$Vol_k(B)Vol_k(C)=\sum_{J\in\mathcal{S}(n,k)}Vol_k(\pi_J(B))Vol_k(\pi_J(C))$.    
\end{center}

Next we prove the Product Formula for Volume for two unions of almost disjoint collections of \textbf{u}-boxes in $V$ and $W$.
Let $\mathcal{B}\in\mathbb{B}_0(V,\textbf{u})$ and $\mathcal{C}\in\mathbb{B}_0(W,\textbf{u})$.
For each $J\in\mathcal{S}(n,k)$, since $\pi_J:\mathbb{E}^n\rightarrow\mathbb{E}^J$ is a linear map, then Corollary 3.10 implies that $Vol_k(\pi_J(\bigcup\mathcal{B}))=\sum_{B\in\mathcal{B}}Vol_k(\pi_J(B))$ and $Vol_k(\pi_J(\bigcup\mathcal{C}))=\sum_{C\in\mathcal{C}}Vol_k(\pi_J(C))$.
Hence, we have:
\begin{center}
$Vol_k(\bigcup\mathcal{B})Vol_k(\bigcup\mathcal{C})=(\sum_{B\in\mathcal{B}}Vol_k(B))(\sum_{C\in\mathcal{C}}Vol_k(C))=$
\end{center}
\medskip\begin{center}
$\sum_{B\in\mathcal{B}}\sum_{C\in\mathcal{C}}Vol_k(B)Vol_k(C)=$
\end{center}
\medskip\begin{center}
$\sum_{B\in\mathcal{B}}\sum_{C\in\mathcal{C}}(\sum_{J\in\mathcal{S}(n,k)}Vol_k(\pi_J(B))Vol_k(\pi_J(C)))=$
\end{center}
\medskip\begin{center}
$\sum_{J\in\mathcal{S}(n,k)}(\sum_{B\in\mathcal{B}}\sum_{C\in\mathcal{C}}Vol_k(\pi_J(B))Vol_k(\pi_J(C)))=$
\end{center}
\medskip\begin{center}
$\sum_{J\in\mathcal{S}(n,k)}(\sum_{B\in\mathcal{B}}Vol_k(\pi_J(B)))(\sum_{C\in\mathcal{C}}Vol_k(\pi_J(C)))=$
\end{center}
\medskip\begin{center}
$\sum_{J\in\mathcal{S}(n,k)}Vol_k(\pi_J(\bigcup\mathcal{B}))Vol_k(\pi_J(\bigcup\mathcal{C}))$.
\end{center}

\medskip To complete the proof of the Product Formula for Volume, it suffices to prove that the function $\Psi:\mathcal{M}(V)\times\mathcal{M}(W)\rightarrow\mathbb{R}$ defined by
\medskip\begin{center}
$\Psi(A,B)=Vol_k(A)Vol_k(B)-\sum_{J\in\mathcal{S}(n,k)}Vol_k(\pi_J(A))Vol_k(\pi_J(B))$ 
\end{center}
equals $0$ at each point $(A,B)\in\mathcal{M}(V)\times\mathcal{M}(W)$.
We argue that $\Psi$ is continuous.
A remark following the definition of the \emph{symmetric difference pseudometric} in section 3 established that measures are continuous.  In particular, the following functions are continuous:
\[
Vol_k:\mathcal{M}(V)\rightarrow\mathbb{R}\text{, }Vol_k:\mathcal{M}(W)\rightarrow\mathbb{R}\text{ and }\]
\[Vol_k:\mathcal{M}(\mathbb{E}^J)\rightarrow\mathbb{R}\text{ for each }J\in\mathcal{S}(n,k).
\]
Since the functions $\pi_J|V:V\rightarrow\mathbb{E}^J$ and $\pi_J|W:W\rightarrow\mathbb{E}^J$ are linear for each $J\in\mathcal{S}(n,k)$, then Corollary 3.6 implies that the following functions are dilations and, hence, are continuous:
\[
(\pi_J|V)_*:\mathcal{M}(V)\rightarrow\mathcal{M}(\mathbb{E}^J)\text{ and } (\pi_J|V)_*:\mathcal{M}(V)\rightarrow\mathcal{M}(\mathbb{E}^J)\text{ for each }J\in\mathcal{S}(n,k).
\]
Consequently, the compositions
\[
Vol_k\circ(\pi_J|V)_*:\mathcal{M}(V)\rightarrow \mathbb{R}\text{ and }Vol_k\circ(\pi_J|W)_*:\mathcal{M}(W)\rightarrow \mathbb{R}
\]
are continuous for each $J\in\mathcal{S}(n,k)$.
Thus, the function
\[
(A,B)\mapsto(Vol_k(A),Vol_k(B),(Vol_k\circ\pi_J(A),Vol_k\circ\pi_J(B))_{J\in\mathcal{S}(n,k)})
\]
\[
:\mathcal{M}(V)\times\mathcal{M}(W)\rightarrow\mathbb{R}^2\times(\mathbb{R}^2)^{|\mathcal{S}(n,k)|}
\]
is continuous, where $|\mathcal{S}(n,k)|$ denotes the number of elements in the finite set $\mathcal{S}(n,k)$.
Also the function
\begin{center}
$(s_V,s_W,(t_{V,J},t_{W,J})_{J\in\mathcal{S}(n,k)})\mapsto s_Vs_W-\sum_{J\in\mathcal{S}(n,k)}t_{V,J}t_{W,J}$
\end{center}
\begin{center}
$:\mathbb{R}^2\times(\mathbb{R}^2)^{|\mathcal{S}(n,k)|}\rightarrow\mathbb{R}$
\end{center}
is continuous.
Since $\Psi$ is the composition of the latter two functions, it is continuous.
Let $\mathcal{D}_V=\{\bigcup\mathcal{B}:\mathcal{B}\in\mathbb{B}_0(V,\textbf{u})\}$ and $\mathcal{D}_W=\{\bigcup\mathcal{C}:\mathcal{C}\in\mathbb{B}_0(W,\textbf{u})\}$.
The preceding proof establishes that $\Psi$ equals $0$ at each point $(\bigcup\mathcal{D},\bigcup\mathcal{E})$ of the subset $\mathcal{D}_V\times\mathcal{D}_W$ of $\mathcal{M}(V)\times\mathcal{M}(W)$.
Since, by Proposition 3.10$^\prime$, $\mathcal{D}_V$ is a dense subset of $\mathcal{M}(V)$ and $\mathcal{D}_W$ is a dense subset of $\mathcal{M}(W)$, then $\mathcal{D}_V\times\mathcal{D}_W$ is a dense subset of $\mathcal{M}(V)\times\mathcal{M}(W)$.
It follows that $\Psi$ equals $0$ at every point of $\mathcal{M}(V)\times\mathcal{M}(W)$.  \qed

\medskip As noted in the Abstract, the Pythagorean Theorem for Volume follows immediately from the Product Formula for Volume by setting $A=B$.

\section{The Binet-Cauchy Formula and the inner product on $\Lambda_k(\mathbb{E}^n)$}
For $1\leq k\leq n$, if $A=(a_{i,j})$ is a $k\times n$ matrix and $J=\{j_1,j_2,\cdots,j_k\}\in\mathcal{S}(n,k)$ where $1\leq j_1<j_2<\cdots<j_k\leq n$, let $A^{[J]}=(a^{[J]}_{i,t})$ denote the $k\times k$ matrix determined by $a^{[J]}_{i,t}=a_{i,j_t}$.
(In other words, for $1\leq t\leq k$, the $t^{th}$ column of $A^{[J]}$ is the $j_t^{th}$ column of $A$.)  For $1\leq k\leq n$, any two $k\times n$ matrices $A$ and $B$ satisfy the equation:

\medskip\noindent\emph{\textbf{The Binet-Cauchy} (or \textbf{Cauchy-Binet}) \textbf{Formula:}}
\begin{center}
$Det(AB^T)=\sum_{J\in\mathcal{S}(n,k)}Det(A^{[J]})Det(B^{[J]})$.
\end{center}
\medskip We will show that the Binet-Cauchy Formula is equivalent to the equation we have called the Product Formula for Inner Products specialized to the inner product $\langle\langle\;,\;\rangle\rangle$ on $\Lambda_k(\mathbb{E}^n)$.

Let $1\leq k\leq n$ and let $A$ and $B$ be $k\times n$ matrices.
For $1\leq i\leq k$, let $x_i=(x_{i,1},x_{i,2},\cdots,x_{i,n})\in\mathbb{E}^n$ be the $i^{th}$ row of $A$ and let $y_i=(y_{i,1},y_{i,2},\cdots,y_{i,n})\in\mathbb{E}^n$ be the $i^{th}$ row of $B$.
Recall that $\textbf{e}=(e_1,e_2,\cdots,e_n)$ denotes the standard ordered orthonormal basis for $\mathbb{E}^n$; and for $J=\{j_1,j_2,\cdots,j_k\}\in\mathcal{S}(n,k)$ where $1\leq j_1<j_2<\cdots<j_k\leq n$, $\textbf{e}_J=(e_{j_1},e_{j_2},\cdots,e_{j_k})$ and $\wedge_J\textbf{e}=e_{j_1}\wedge e_{j_2}\wedge\cdots\wedge e_{j_k}$.
The Product Formula for the inner product $\langle\langle\;,\;\rangle\rangle$ on $\Lambda_k(\mathbb{E}^n)$ implies:

\medskip\begin{center}
$\langle\langle x_1\wedge x_2\wedge\cdots\wedge x_k,y_1\wedge y_2\wedge\cdots\wedge y_k\rangle\rangle=$
\end{center}
\begin{center}
$\sum_{J\in\mathcal{S}(n,k)}\langle\langle x_1\wedge x_2\wedge\cdots\wedge x_k,\wedge_J\textbf{e}\rangle\rangle\langle\langle y_1\wedge y_2\wedge\cdots\wedge y_k,\wedge_J\textbf{e}\rangle\rangle$.
\end{center}
Thus, in order to convert this equation into the Binet-Cauchy Formula, it suffices to prove that 
\[
\langle\langle x_1\wedge x_2\wedge\cdots\wedge x_k,y_1\wedge y_2\wedge\cdots\wedge y_k\rangle\rangle=Det(AB^T), 
\]
and
\[\langle\langle x_1\wedge x_2\wedge\cdots\wedge x_k,\wedge_J\textbf{e}\rangle\rangle=Det(A^{[J]})\text{ and }\langle\langle y_1\wedge y_2\wedge\cdots\wedge y_k,\wedge_J\textbf{e}\rangle\rangle=Det(B^{[J]})\] 
for every $J\in\mathcal{S}(n,k)$.
Observe that
\[
\langle\langle x_1\wedge x_2\wedge\cdots\wedge x_k,y_1\wedge y_2\wedge\cdots\wedge y_k\rangle\rangle=Det(x_i\sbullet y_j),
\]
\[
\langle\langle x_1\wedge x_2\wedge\cdots\wedge x_k,\wedge_J\textbf{e}\rangle\rangle=\langle\langle x_1\wedge x_2\wedge\cdots\wedge x_k,e_{j_1}\wedge e_{j_2}\wedge\cdots\wedge e_{j_k}\rangle\rangle=Det(x_i\sbullet e_{j_t})
\]
and
\[\langle\langle y_1\wedge y_2\wedge\cdots\wedge y_k,\wedge_J\textbf{e}\rangle\rangle=\langle\langle y_1\wedge y_2\wedge\cdots\wedge y_k,e_{j_1}\wedge e_{j_2}\wedge\cdots\wedge e_{j_k}\rangle\rangle=Det(y_i\sbullet e_{j_t})
\]
for $J=\{j_1,j_2,\cdots,j_k\}\in\mathcal{S}(n,k)$ where $1\leq j_1<j_2<\cdots<j_k\leq n$.
Hence, converting the Product Formula for $\langle\langle\;,\;\rangle\rangle$ to the Binet-Cauchy Formula is simply a matter of comparing the entries of the matrices $(x_i\sbullet y_j)$, $(x_i\sbullet e_{j_t})$ and $(y_i\sbullet e_{j_t})$ to the corresponding entries of the matrices $AB^T$, $A^{[J]}$ and $B^{[J]}$.
Clearly, the $(i,j)^{th}$ entries of $(x_i\sbullet y_j)$ and $AB^T$ are both equal to $x_i\sbullet y_j$, the $(i,t)^{th}$ entries of $(x_i\sbullet e_{j_t})$ and $A^{[J]}$ are both equal to $x_{i,j_t}$, and the $(i,t)^{th}$ entries of $(y_i\sbullet e_{j_t})$ and $B^{[J]}$ are both equal to $y_{i,j_t}$.
We conclude that the Binet-Cauchy Formula is simply a reformulation of the Product Formula for $\langle\langle\;,\;\rangle\rangle$.  \qed

\medskip We note that one consequence of the Binet-Cauchy Formula is the following occasionally useful fact.

\begin{corollary}
For $1\leq k\leq n$, a $k\times n$ matrix $A$ has rank $k$ if and only if $Det(A^{[J]})\neq 0$ for some $J\in\mathcal{S}(n,k)$.
\end{corollary}

\begin{proof}
We set $B=A$ in the Binet-Cauchy Formula to obtain the equation

\begin{center}
$Det(AA^T)=\sum_{J\in\mathcal{S}(n,k)}(Det(A^{[J]}))^2$.
\end{center}

\noindent Then with the help of Lemma 6.2 below, we obtain the equivalence of the following three statements.
\textit{i})\;$rank(A)=k$.
\textit{ii})\;$Det(AA^T)\neq 0$.
\textit{iii})\;$Det(A^{[J]})\neq 0$ for some $J\in\mathcal{S}(n,k)$.
\end{proof}

\begin{lemma}
For $1\leq k\leq n$, a $k\times k$ matrix $A$ has rank $k$ if and only if $Det(AA^T)\neq 0$.
\end{lemma}
\begin{proof}
Consider the following statements.  
\textit{i})\;$rank(A)<k$.  
\textit{ii})\;The rows of A are linearly dependent.  
\textit{iii})\;There is an $a\in\mathbb{E}^k-\{0_k\}$ such that $aA=0_n$ (where $0_k$ and $0_n$ are the additive identities of $\mathbb{E}^k$ and $\mathbb{E}^n$, respectively).
\textit{iv})\;There is an $a\in\mathbb{E}^k-\{0_k\}$ such that $a(AA^T)a^T=(aA)\sbullet(aA)=0$.
\textit{v})\;There is an $a\in\mathbb{E}^k-\{0_k\}$ such that $a(AA^T)=0_k$.
\textit{vi})\;The rows of $AA^T$ are linearly dependent.
\textit{vii})\;$Det(AA^T)=0$.
Observe that \textit{i}), \textit{ii}) and \textit{iii}) are equivalent, \textit{iii}) $\implies$ \textit{v}) $\implies$ \textit{iv}) $\implies$ \textit{iii}), and \textit{v}), \textit{vi}) and \textit{vii}) are equivalent.
\end{proof}



\printbibliography

\end{document}